%% file: main.tex
\documentclass[final]{article}

\usepackage{lmodern}
\usepackage{amsmath}
\usepackage{amssymb}
\usepackage{amsthm}
\usepackage{dsfont,upgreek}

\usepackage{todonotes}
\usepackage{enumitem}
\usepackage[english]{babel}
\usepackage[babel]{csquotes}

\usepackage{xcolor,graphicx,tikz}
\usepackage{ifthen}
\usepackage{cite}

\usepackage[bookmarks=true,bookmarksnumbered,plainpages=false,linktocpage,colorlinks=true,citecolor=green!80!black,linkcolor=red!70!black,filecolor=magenta,urlcolor=magenta,hidelinks,breaklinks,pdfauthor={Thomas Jahn, Tino Ullrich, Felix Voigtlaender},pdftitle={Sampling numbers of smoothness classes via ell-1-minimization},unicode=true]{hyperref}
\usepackage[nameinlink,sort&compress,capitalize]{cleveref}

\usetikzlibrary{calc,intersections,through,backgrounds,arrows,patterns,shapes.geometric,shapes.misc,spy,3d,cd}

\numberwithin{equation}{section}

\theoremstyle{plain}
\newtheorem{theorem}{Theorem}[section]
\newtheorem{lemma}[theorem]{Lemma}
\newtheorem{corollary}[theorem]{Corollary}

\newtheorem{assumption}[theorem]{Assumption}
\theoremstyle{definition}
\newtheorem{definition}[theorem]{Definition}
\newtheorem{remark}[theorem]{Remark}

\binoppenalty=\maxdimen
\relpenalty=\maxdimen
\newcommand{\N}{\mathbb{N}}
\newcommand{\Z}{\mathbb{Z}}

\newcommand{\R}{\mathbb{R}}
\newcommand{\T}{\mathbb{T}}
\newcommand{\CC}{\mathbb{C}}
\newcommand{\ee}{\mathrm{e}}
\newcommand{\ii}{\mathrm{i}}
\newcommand{\dd}{\mathrm{d}}

\newcommand{\CalA}{\mathcal{A}}
\newcommand{\CalB}{\mathcal{B}}

\newcommand{\CalF}{\mathcal{F}}
\newcommand{\CalK}{\mathcal{K}}
\newcommand{\CalP}{\mathcal{P}}
\newcommand{\CalT}{\mathcal{T}}

\newcommand{\bfun}{L^\infty}
\newcommand{\gp}{p}

\newcommand{\eps}{\varepsilon}
\newcommand{\FirstN}[1]{\underline{#1}}

\newcommand{\fall}{\:\forall\:}
\newcommand{\abs}[1]{\left\lvert#1\right\rvert}
\newcommand{\mnorm}[1]{\left\lVert#1\right\rVert}
\newcommand{\setn}[1]{\left\{#1\right\}}
\newcommand{\setcond}[2]{\left\{#1 \,:\, #2\right\}}
\newcommand{\defeq}{\mathrel{\mathop:}=}
\newcommand{\lr}[1]{\left(#1\right)}
\newcommand{\skpr}[2]{\left\langle#1,#2\right\rangle}
\newcommand{\norel}{\mathrel{\phantom{=}}}
\newcommand{\card}[1]{\left\lvert#1\right\rvert}
\newcommand{\ceil}[1]{\left\lceil#1\right\rceil}
\newcommand{\floor}[1]{\left\lfloor#1\right\rfloor}

\DeclareMathOperator{\linspan}{span}
\DeclareMathOperator{\codim}{codim}
\DeclareMathOperator{\rank}{rank}
\DeclareMathOperator{\co}{co}
\DeclareMathOperator{\id}{id}

\newcommand{\Log}[1]{\log(#1+1)}
\newcommand{\pmone}{[-1,1]}

\newcommand{\footremember}[2]{%
    \footnote{#2}
    \newcounter{#1}
    \setcounter{#1}{\value{footnote}}%
}
\newcommand{\footrecall}[1]{%
    \footnotemark[\value{#1}]%
}

\makeatletter
\newcommand{\subjclass}[2][2020]{%
  \let\@oldtitle\@title%
  \gdef\@title{\@oldtitle\footnotetext{#1 \emph{Mathematics subject classification.} #2}}%
}
\newcommand{\keywords}[1]{%
  \let\@@oldtitle\@title%
  \gdef\@title{\@@oldtitle\footnotetext{\emph{Key words and phrases.} #1.}}%
}
\makeatother

\begin{document}

\title{Sampling numbers of smoothness classes via $\ell^1$-minimization}
\author{Thomas Jahn\footremember{tj}{
Mathematical Institute for Machine Learning and Data Science (MIDS), Catholic University of Eichstätt--Ingolstadt (KU), Auf der Schanz 49, 85049 Ingolstadt, Germany.
Email: \href{mailto:thomas.jahn@ku.de}{\texttt{thomas.jahn@ku.de}}, \href{mailto:felix.voigtlaender@ku.de}{\texttt{felix.voigtlaender@ku.de}}
} \and Tino Ullrich\footremember{tu}{
Technische Universität Chemnitz, 09107 Chemnitz, Germany.
Email: \href{mailto:tino.ullrich@mathematik.tu-chemnitz.de}{\texttt{tino.ullrich@mathematik.tu-chemnitz.de}}
} \and Felix Voigtlaender\footrecall{tj}}


\maketitle

\begin{abstract}
Using techniques developed recently in the field of compressed sensing
we prove new upper bounds for general (nonlinear) sampling numbers of (quasi-)Banach smoothness spaces in $L^2$.
In particular, we show that in relevant cases such as mixed and isotropic weighted Wiener classes or Sobolev spaces with mixed smoothness, sampling numbers in $L^2$ can be upper bounded by best $n$-term trigonometric widths in $\bfun$.
We describe a recovery procedure from $m$ function values based on $\ell^1$-minimization (basis pursuit denoising).
With this method, a significant gain in the rate of convergence compared to recently developed linear recovery methods is achieved.
In this deterministic worst-case setting we see an additional speed-up of $m^{-1/2}$ (up to log factors) compared to linear methods in case of weighted Wiener spaces.
For their quasi-Banach counterparts even arbitrary polynomial speed-up is possible.
Surprisingly, our approach allows to recover mixed smoothness Sobolev functions belonging to $S^r_pW(\T^d)$ on the $d$-torus with a logarithmically better rate of convergence than any linear method can achieve
when $1 < p < 2$ and $d$ is large.
This effect is not present for isotropic Sobolev spaces. 
\end{abstract}

\input{Introduction.tex}

\input{AbstractResult.tex}

\input{PolynomialBases.tex}

\input{Applications.tex}

\input{UnboundedONS.tex}

\appendix

\input{MixedSobolev.tex}

\parindent 0pt
\textbf{Acknowledgement.} The authors thank Dinh D\~ung, David Krieg, Erich Novak, Winfried Sickel, Serhii Stasyuk, Vladimir Temlyakov, and Mario Ullrich for helpful comments. 
FV and TJ acknowledge support by the Hightech Agenda Bavaria.
This research is co-financed with tax revenues on the basis of the budget passed by the Saxon state parliament, Germany (grant number SAB 100403339; project NutriCon).

\footnotesize
\bibliographystyle{plain}
\bibliography{references.bib}
\end{document}

%% file: Introduction.tex
\section{Introduction}%
\label{sec:Introduction}

In this paper we study the recovery problem for multivariate functions
belonging to a given smoothness class $\CalF$ using only $m$ function values.
This topic gained a lot of interest in recent years and yet several important questions remain open.
In a series of papers \cite{BSU22,DKU22,KUV19,KriegUl2021a,KriegUl2021b, NagelScUl2022,Temlyakov2021,TemlyakovUl2021,TemlyakovUl2022} several
authors made significant progress towards \emph{linear} sampling recovery,
where the recovery operator is supposed to be linear.
Surprisingly, in most of the relevant cases it turned out that linear sampling recovery
(quantified by the sampling numbers) is as powerful as general linear approximation
(quantified by the approximation numbers, see \Cref{sect:widths} below). 

However, the picture may change drastically if we allow for \emph{nonlinear} sampling recovery,
where we still use $m$ function values but allow the recovery operator to be nonlinear.
It is well-known, see \cite[Propositions~13 and~14]{CreutzigWo2004}
and \cite[Theorem~4.8]{NovakWo2008} and the references therein, that we have to focus on examples
where neither the smoothness class $\CalF$ is a Hilbert space nor the target space in which the recovery error is measured is given by $X=\bfun$ to get a possible gain in the convergence rate.
Nevertheless, our proposed method makes also sense in the mentioned settings since existing approaches
are often far from being implementable or constructive \cite{DKU22,NagelScUl2022} due to the use of a non-constructive subsampling strategy coming from the Kadison--Singer problem or simply from the fact that optimal subspaces for linear approximation in $\bfun$ are only known to exist \cite{Temlyakov2021}. 

Here we study the following quantities which relate to the worst-case setting in information-based complexity using \emph{standard information}, see \cite[Sections~4.1.1 and~4.1.4]{NovakWo2008}.
We define the (general nonlinear) sampling numbers
\begin{equation*}
  \varrho_m (\CalF)_X
  \defeq \inf_{t_1,\ldots,t_m \in \Omega}\,
       \inf_{R : \CC^m \to X}\,
         \sup_{\mnorm{f}_{\CalF} \leq 1}\,
           \mnorm{f - R(f(t_1),\ldots,f(t_m))}_X
\end{equation*}
for a quasi-normed space $\CalF$ of functions $\Omega\to\CC$
which is continuously embedded into the Banach space $\bfun(\Omega)$ of bounded (measurable) functions.
(Note that we use the same symbol as for essentially bounded functions here.)
As a concrete recovery method we apply a convex optimization method, namely $\ell^1$-minimization, which is a popular method in statistics \cite{Wainwright2019},
statistical learning, and in the theory of compressed sensing \cite{RauhutCompressiveSensingBook}.
The idea to apply this method to function recovery problems is not new
and has been already successfully applied in \cite{AdcockCaDeMo2022,RauhutSc2017,RauhutWard2012,RauhutWa2016,FV22}.
The general recovery method is rather simple and uses two main ingredients from compressed sensing:
(a) recent RIP results for matrices coming from bounded orthonormal systems \cite{Bourgain2014,HavivRe2017} and
(b) guarantees for $\ell^1$-minimization with noise \cite{RauhutCompressiveSensingBook}.
Crucial for our analysis is the parameter choice in the optimization program
called \emph{basis pursuit denoising}, which ensures that we recover a near-optimal $n$-term approximation.
This technique has been applied recently to the well-known Barron class by the third-named author, see \cite{FV22}. 

In this paper we give an analysis of this basis pursuit denoising method
in a much more general context and relate the $L^2$ recovery error
mainly to the best $n$-term approximation error in $\bfun(\T^d)$.
A specific instance of our general result reads as follows:
There are universal constants $C, \widetilde{C}>0$ such that for any $d \in \N$, and any quasi-normed space $\CalF \hookrightarrow \bfun(\T^d)$,
and arbitrary $n,M \in \N$ with $M \geq 3$, we have
\begin{equation}\label{eq:intro1}
  \begin{split}
     &\varrho_{\ceil{Cd \Log{d} n \Log{n}^2\log(M)}} (\CalF)_{L^2(\T^d)}\\ 
     &~~~~\leq \widetilde{C}
           \lr{
                   \sigma_n(\CalF;\CalT^d)_{\bfun(\T^d)}
                   + E_{[-M,M]^d\cap \Z^d} (\CalF;\CalT^d)_{\bfun(\T^d)}
                 },
	\end{split}
\end{equation}
see \Cref{thm:improvedGeneralFourierBound}.
Here, 
\begin{equation*}
\sigma_n(\CalF;\CalT^d)_{\bfun(\T^d)}=\sup_{\substack{f\in\CalF,\\\mnorm{f}_\CalF\leq 1}}\,\inf_{\substack{J\subset\Z^d,\card{J}\leq n,\\(c_k)_{k\in J}\in\CC^J}}\mnorm{f-\sum_{k\in J}  c_k\ee^{2\uppi\ii\skpr{k}{\bullet}}}_{\bfun(\T^d)}
\end{equation*}
denotes the best $n$-term approximation of $\CalF$ width with respect to the multivariate trigonometric system $\CalT^d$,
whereas the quantity 
\begin{equation*}
E_J (\CalF;\CalT^d)_{\bfun(\T^d)}=\sup_{\substack{f\in\CalF,\\\mnorm{f}_\CalF\leq 1}}\,\inf_{(c_k)_{k\in J}\in\CC^J}\mnorm{f-\sum_{k\in J}c_k\ee^{2\uppi\ii\skpr{k}{\bullet}}}_{\bfun(\T^d)}
\end{equation*}
denotes the best worst-case trigonometric approximation error for $\CalF$ and given frequency set $J \subset\Z^d$,
both measured in $\bfun(\T^d)$.
The asymptotic quantities on the right-hand side have been extensively studied in the last decades for hyperbolic cross spaces, see \cite[Chapters~4 and~7]{DungTeUl2018} and the references therein.
Note that the parameter $M$ determining the box size only enters logarithmically in the index of the left-hand side quantity.
As already mentioned, the result is an instance of a more general framework
involving uniformly bounded dictionaries $\CalB$, see \Cref{thm:MainAbstractResult}
and the subsequent \Cref{Cor:general_sampling}.
In a sense, the relation \eqref{eq:intro1} goes in a similar direction
as Temlyakov's recent observation \cite{Temlyakov2021},
where the linear sampling number is related to the Kolmogorov width in $\bfun(\T^d)$.
In fact, choosing $M$ as an appropriate power of $n$ (depending on the class $\CalF$)
the right-hand side is essentially dominated by the behavior of the best $n$-term width
$\sigma_n(\CalF;\CalT^d)_{\bfun(\T^d)}$.

The advantage of the nonlinear method is already visible for the Sobolev spaces
of mixed smoothness $S^r_pW(\T^d)$ in $L^2(\T^d)$ if $1<p<2$ and $r>\frac{1}{p}$.
In general, we know that the Gelfand numbers for this Sobolev embedding are asymptotically smaller
than the approximation numbers, see \cite{ByrenheidUl2017,Vy08}.
Hence, linear algorithms are not optimal for the worst-case approximation of these classes.
For the sampling numbers we obtain from \eqref{eq:intro1} in case $1<p<\infty$ and $r>\frac{1}{p}$
\begin{equation*}
  \varrho_{\ceil{C_{r,p,d}n\Log{n}^3}}(S^r_pW(\T^d))_{L^2(\T^d)}
  \lesssim\sigma_n(S^r_pW(\T^d);\CalT^d)_{\bfun(\T^d)}.
\end{equation*}
Thanks to Temlyakov \cite[Theorem~2.9]{Temlyakov2015}, see also \cite[Theorem~7.5.2]{DungTeUl2018},
it is known that 
\begin{equation}\label{eq:intro2}
	\sigma_n(S^r_pW(\T^d);\CalT^d)_{\bfun(\T^d)}
	\lesssim \lr{\frac{\Log{n}^{d-1}}{n}}^{r-\frac{1}{p}+\frac{1}{2}}\Log{n}^{\frac{1}{2}-(d-1)(\frac{1}{p}-\frac{1}{2})}
\end{equation}
when $1 < p < 2$ and $\frac{1}{p} < r$.
This together with \Cref{lem:technical} shows that for the approximation numbers $a_n$ and the Kolmogorov numbers $d_n$ we have that $\varrho_n = o(a_n)$ as $n\to\infty$ for sufficiently large dimensions $d$ since 
\begin{equation*}
	a_n(S^r_pW(\T^d))_{L^2(\T^d)}
  = d_n(S^r_pW(\T^d))_{L^2(\T^d)}
  \asymp \lr{\frac{\Log{n}^{d-1}}{n}}^{r-\frac{1}{p}+\frac{1}{2}},
\end{equation*}
see \cite[Theorem~4.5.1]{DungTeUl2018}.
Note that the acceleration happens in the $d$-dependent second $\log$-factor in \eqref{eq:intro2}.
Therefore, in the context of Sobolev spaces with mixed smoothness linear and nonlinear sampling numbers might show different orders of decay.
This phenomenon is not present for usual isotropic Sobolev spaces as shown by Heinrich \cite{Heinrich2009} which confirms the conjecture in \cite[Open Problem~18]{NovakWo2008}.

Another source of examples is provided by the weighted Wiener-type spaces $\CalA^r$ and their mixed counterparts $\CalA^r_{\mathrm{mix}}$,
which gained interest in recent years, see \cite{KaPoVo15,KLT21}.
The classical trigonometric spaces on the $d$-torus are defined as 
\begin{equation*}
	\CalA^r_{\mathrm{mix}}(\T^d)
  \defeq \setcond{f:\T^d\to\CC}
                 {\sum_{k\in\Z^d}\abs{\widehat{f}(k)}\prod_{j=1}^d(1+\abs{k_j})^r <\infty}.
\end{equation*}
As a direct consequence of the general bound \eqref{eq:intro1} together with known behavior
of the best $n$-term widths we obtain
\begin{equation*}
  \varrho_{\ceil{C_{r,d}n\Log{n}^3}}(\CalA^r_{\mathrm{mix}}(\T^d))_{L^2(\T^d)}
  \lesssim n^{-r-\frac{1}{2}}\Log{n}^{(d-1)r+\frac{1}{2}}.
\end{equation*}
Note that the Kolmogorov numbers (approximation numbers) in this situation lead to a main rate of $n^{-r}$,
see \cite{NguyenNgSi2022}, such that our nonlinear recovery operator shows
an acceleration of $n^{-\frac{1}{2}}$ in the main rate.\footnote{By \emph{main rate}, we mean the polynomial decay, ignoring logarithmic factors.}
This is already best possible, since the acceleration of the nonlinear samping numbers compared to their linear counterparts cannot exceed $n^{-\frac{1}{2}}$, as implied by the results in \cite{CreutzigWo2004}, see also \cite[Section~4.2]{NovakWo2008}. 

Spaces of this type can also be built upon systems of orthogonal polynomials.
Here we additionally study the quasi-Banach setting which leads to interesting observations and speed-up beyond $n^{-\frac{1}{2}}$ as the results below show.
Such spaces for univariate Legendre polynomials have been studied in \cite{RauhutWard2012}.
Here we prove a result for spaces
\begin{equation*}
\CalA^r_{\alpha,p}(\pmone)
 \defeq \setcond{\sum_{n\in\N_0} \beta_n \gp^\alpha_n }{\sum_{n\in \N_0} \abs{\beta_n(1+n)^r}^p<\infty},
 \end{equation*}
where $0<p\leq 1$, $r>0$, $(\gp^\alpha_n)_{n\in\N_0}$ is the system of Chebyshev polynomials of the first kind ($\alpha=-\frac{1}{2}$) or the system of Legendre polynomials ($\alpha=0$), and $\mu_\alpha$ is the probability measure on $\pmone$ with respect to which the polynomials $(\gp_n)_{n\in\N_0}$ are orthonormal.
The parameter restriction $r>0$ ensures that $\CalA^r_{-\frac{1}{2},p}(\pmone)$ is embedded into $C(\pmone)$. 
For Legendre expansions, this is only true if $r>\frac{1}{2}$.
However, for all $r>0$ we have that $x\mapsto (1-x^2)^{\frac{1}{4}}f(x)$ is a continuous and bounded function on $\pmone$ and thus admits evaluation.
In fact, the authors in \cite{RauhutWard2012} managed to overcome the issue that the $L^2$-normalized Legendre polynomials are not uniformly $\bfun$-bounded.  
Using a change of measure technique allows for a version of the basis pursuit denoising method also for Legendre expansions.
Combined with the analysis in this paper, we observe
\begin{equation*}
	\varrho_{\ceil{C_{r,p} n\Log{n}^4}}(\CalA^r_{\alpha,p}(\pmone))_{L^2(\mu_{\alpha})}
  \leq \widetilde{C}_{r,p}n^{-R(r,p,\alpha)}.
\end{equation*}
where $R(r,p,-\frac{1}{2}) = r + \frac{1}{p}-\frac{1}{2}$ in case of Chebyshev polynomials
and\linebreak $R(r,p,0)=r +\frac{1}{p}-1$ in case of Legendre polynomials, with constants $C_{r,p}>0$ and $\widetilde{C}_{r,p}>0$.
Here we see that we have a speed-up of $n^{-\frac{1}{p}+\frac{1}{2}}$ and $n^{-\frac{1}{p}+1}$ compared to the approximation numbers, respectively. This shows that for quasi-Banach spaces we are not limited to an acceleration of order $n^{-\frac{1}{2}}$, see \Cref{thm:convex} and \Cref{thm:approximation-chebyshev-wiener}.\\\newline
\textbf{Notation.} As usual, the symbol $\N\defeq \setn{1,2,3,\ldots}$ denotes the natural numbers, $\N_0\defeq\N\cup\setn{0}$, $\Z$ denotes the integers, $\R$ the real numbers, and $\CC$ the complex numbers.
The set $\setcond{k\in\N}{1\leq k\leq n}$ shall be abbreviated by $\FirstN{n}$.
We write $\log$ for the natural logarithm.
The symbol $\CC^{m\times n}$ denotes the set of all $m\times n$-matrices with complex entries.
Quasi-normed spaces $(X,\mnorm{\bullet}_X)$ consist of a vector space $X$ and a function $\mnorm{\bullet}_X:X\to\R$ that takes only non-negative values, vanishes only at $0\in X$, is absolutely homogeneous, and satisfies the quasi-triangle inequality $\mnorm{x+y}_X\leq C(\mnorm{x}_X+\mnorm{y}_X)$ for all $x,y\in X$ for some absolute constant $C>0$.
The closed unit ball in such a space will be denoted by $B_X$.
If $(X,\mnorm{\bullet}_X)$ is even a Hilbert space, the inner product in $X$ is denoted by $\skpr{\bullet}{\bullet}_X$.
Particular instances of quasi-normed spaces are the spaces $\ell^p(A)\defeq L^p(A,\#)$ of $p$-integrable functions with respect to the counting measure $\#$ on a finite or countably infinite set $A$.
If $A=\setn{1,\ldots,m}$, we write $\ell^p(m)$ for $\ell^p(A)\cong \CC^m$, and denote the inner product in $\ell^2(m)$ by $\skpr{\bullet}{\bullet}$ without any indices.
For two sequences $(a_n)_{n\in\N}$ and $(b_n)_{n\in\N}$ of non-negative real numbers, we write $a_n\lesssim b_n$ if there exists a constant $C>0$ such that $a_n\leq Cb_n$.
We write $a_n\gtrsim b_n$ if $b_n\lesssim a_n$, and $a_n\asymp b_n$ if both $a_n\lesssim b_n$ and $b_n\lesssim a_n$.
In case $\lim_{n\to \infty} a_n/b_n = 0$ we sometimes write $a_n = o(b_n)$.
Even though the implied constants do not depend on $n$, they may depend on other parameters.
In our cases, these could be the dimension $d$ of domain, the smoothness parameter $r$, or the integrability parameter $p$.
Sometimes, we will spell out the dependence of certain parameters by indexing the constants, i.e., $C_p$ will be a positive number that may depend on $p$ but not on $n$.
The convex hull $\co(A)$ of a subset $A$ of a real vector space is the set $\setcond{\sum_{j=1}^n\lambda_j x_j}{n\in\N,\lambda_j\in\R,\lambda_j\geq 0,\sum_{j=1}^n\lambda_j=1,x_j\in A}$.

%% file: AbstractResult.tex
\section{A general recovery result}%
\label{sec:AbstractResult}

\subsection{Best \texorpdfstring{$n$}{n}-term approximation}
\label{sec:best-n-term}
\begin{assumption}\label{assu:StandingAssumptions}
  Let $(\Omega,\mu)$ be a probability space and let $\CalB = (b_j)_{j \in I} \subset L^\infty(\mu)$
  be a \emph{bounded orthonormal system}, i.e., $\CalB$ is orthonormal in $L^2(\mu)$,
  and there exists some $K = K(\CalB) > 0$ such that $\mnorm{b_j}_{\bfun(\Omega)} \leq K$ for all $j \in I$.
  The family $\CalB$ will also be referred to as a \emph{dictionary}.
  Here, we denote by $\bfun(\Omega)$ the space of all bounded measurable functions
  $f:\Omega \to \CC$ (i.e., we do \emph{not} identify functions agreeing $\mu$-almost everywhere)
  with norm $\mnorm{f}_{\bfun(\Omega)} \defeq \sup_{x \in \Omega} \abs{f(x)}$.
\end{assumption}

For $n \in \N$, we define the set of linear combinations of $n$ elements of $\CalB$ as
\begin{equation*}
\Sigma_n\defeq  \Sigma_n(\CalB)\defeq\setcond{\sum_{j \in J}c_jb_j}{J \subset I, \card{J} \leq n, (c_j)_{j \in J} \in \CC^J}.
\end{equation*}
Furthermore, given $J \subset I$, we denote the linear span of $(b_j)_{j \in J}$ by
\begin{equation*}
  V_J \defeq V_J(\CalB)\defeq \linspan \setcond{b_j}{j \in J}.
\end{equation*}
In this paper, $\mu$ is assumed to be a probability (and therefore finite) measure, which implies that
the usual equivalence class $[f]_{\mu}$ under identification of functions agreeing $\mu$-almost everywhere
belongs to $L^2(\mu)$ for all $f\in\bfun(\Omega)$.
This gives an embedding (i.e., a continuous map) $\bfun(\Omega) \to L^2(\mu)$
which allows to view $\bfun(\Omega)$ as a subset of $L^2(\mu)$ up to identification of functions that agree $\mu$-almost everywhere.
We write $\bfun(\Omega)\hookrightarrow L^2(\mu)$ for this situation.

Let $X$ be a normed space of functions on $\Omega$ (e.g., $X = \bfun(\Omega)$ or $X = L^2(\mu)$) and $f\in X$.
We denote by 
\begin{equation*}
  \sigma_n (f; \CalB)_X
  \defeq \inf_{g \in \Sigma_n}\mnorm{f - g}_X
\end{equation*}
and for fixed $J \subset I$, 
\begin{equation*}
  E_J(f;\CalB)_X
  \defeq \inf_{g \in V_J}\mnorm{f - g}_X
\end{equation*}
the corresponding \emph{best approximation errors} from the nonlinear set $\Sigma_n$
and the linear space $V_J$, respectively. 

For a non-empty quasi-normed space $(\CalF,\mnorm{\bullet}_\CalF)$ of functions on $\Omega$
which embeds into $X$, we denote the corresponding \emph{worst-case errors} by 
\begin{equation*}
  \sigma_n (\CalF; \CalB)_X
  \defeq\sup_{f\in B_\CalF} \sigma_n (f; \CalB)_X
\end{equation*}
and 
\begin{equation*}
  E_J(\CalF;\CalB)_X
  \defeq \sup_{f\in B_\CalF}E_J(f;\CalB)_X.
\end{equation*}

Note that the approximation performed in the definition of $\sigma_n(f;\CalB)_X$ is considered \emph{nonlinear} since the set $\Sigma_n$ is not a vector space.
However, we still have a scaling property in the sense that for arbitrary $f\in \CalF$ (not necessarily from the unit ball)
\begin{equation}\label{eq:scaling}
	\sigma_n(f;\CalB)_X \leq \sigma_n (\CalF; \CalB)_X\mnorm{f}_{\CalF}.
\end{equation}

The following useful relation follows from a straightforward computation but does not appear in the literature in this form at least to the authors' knowledge.
Therefore we explicitly state it here.
We consider the following embedding situation
\tikzcdset{arrow style=tikz, diagrams={>=stealth}}
\begin{equation*}
\begin{tikzcd}
\CalF \ar[hookrightarrow]{rr}{} \ar[hookrightarrow]{rd}{} && X \\
& \mathcal{H} \ar[swap,hookrightarrow]{ru}{} &              
\end{tikzcd}
\end{equation*}
for quasi-normed spaces $\CalF$, $\mathcal{H}$ and the target space $X$, all containing the dictionary $\CalB$.
\begin{lemma}\label{lem:factor}
Let $\CalF$, $\mathcal{H}$ and $X$ as above.
Then it holds for $n_1,n_2 \in \N$
\begin{equation*}
	\sigma_{n_1+n_2}(\CalF;\CalB)_X \leq \sigma_{n_1}(\CalF;\CalB)_{\mathcal{H}}\cdot \sigma_{n_2}(\mathcal{H};\CalB)_X.
\end{equation*}
\end{lemma}

\begin{proof}
Let $\eps>0$ and $f\in B_\CalF$.
There exists an element $g_{n_1}\in \Sigma_{n_1}$ such that 
\begin{equation*}
\mnorm{f-g_{n_1}}_{\mathcal{H}}\leq \sigma_{n_1}(f;\CalB)_{\mathcal{H}}+\eps.
\end{equation*}
Let us approximate $f-g_{n_1}$ in $X$.
By definition and the above scaling property \eqref{eq:scaling} we know that there exists $g_{n_2} \in \Sigma_{n_2}$ such that 
\begin{equation*}
	\mnorm{f-g_{n_1}-g_{n_2}}_X \leq \sigma_{n_2}(\mathcal{H};\CalB)_X\mnorm{f-g_{n_1}}_{\mathcal{H}}+\eps.
\end{equation*}
Plugging in the bound for $\mnorm{f-g_{n_1}}_{\mathcal{H}}$ yields
\begin{equation*}
 \mnorm{f-g_{n_1}-g_{n_2}}_X \leq (\sigma_{n_1}(f;\CalB)_{\mathcal{H}}+\eps)\cdot (\sigma_{n_2}(\mathcal{H};\CalB)_X+\eps).
\end{equation*}
Since $g_{n_1}+g_{n_2} \in \Sigma_{n_1+n_2}$ we obtain 
\begin{equation*}
 \sigma_{n_1+n_2}(f;\CalB)_X \leq (\sigma_{n_1}(f;\CalB)_{\mathcal{H}}+\eps)\cdot (\sigma_{n_2}(\mathcal{H};\CalB)_X+\eps).
\end{equation*}
Finally, take the limit $\eps\downarrow 0$ and the supremum over $f$ with $\mnorm{f}_{\CalF} \leq 1$. 
\end{proof}

In this paper we aim to relate sampling recovery errors to the above defined quantities $\sigma_n$ and $E_J$.
Given $m \in \N$, we denote the \emph{optimal worst-case nonlinear sampling recovery error}
with respect to $\CalF\hookrightarrow X$ using $m$ point samples by
\begin{equation*}
  \varrho_m (\CalF)_X
  \defeq \inf_{t_1,\ldots,t_m \in \Omega}\,
       \inf_{R : \CC^m \to X}\,
         \sup_{f\in B_\CalF}\,
           \mnorm{f - R(f(t_1),\ldots,f(t_m))}_X,
\end{equation*}
see \cite{Dung1991}.
If we additionally assume the linearity of the reconstruction map $R:\CC^m \to X$ we speak of \emph{linear sampling numbers} denoted by $\varrho^{\mathrm{lin}}_m (\CalF)_X$.
The linear sampling numbers for $X=L^2$ are quite well understood, see \cite{DKU22,NagelScUl2022}.
Note that in the definition of the sampling numbers the dictionary $\CalB$ does not play a role.

The following notion of quasi-projections will be crucial for our results.
Examples for such operators are given by certain de la Vallée Poussin operators, see \Cref{sec:FourierBasis}.
\begin{definition}\label{def:QuasiProjection}
  A \emph{$(\kappa,n,J,J^\ast,\tau)$ quasi-projection} (where $\kappa, n \in \N$, $J, J^\ast \subset I$, and $\tau>0$) is a linear operator $P : \bfun(\Omega) \to \bfun(\Omega)$
  which satisfies the following conditions:
  \begin{enumerate}[label={(\roman*)},leftmargin=*,align=left,noitemsep]
    \item{$P(\Sigma_n) \subset \Sigma_{\kappa n}$},
    \item{$P f = f$ for all $f \in V_J$,\label{fixpoint}}
    \item{$P f \in V_{J^\ast}$ for all $f \in \bfun(\Omega)$,\label{range}}
    \item{$P : \bfun(\Omega) \to \bfun(\Omega)$ is well-defined and bounded,
          with operator norm at most $\tau$, i.e.,
          \begin{equation*}
           \mnorm{P}_{\bfun(\Omega) \to \bfun(\Omega)} \leq \tau.
          \end{equation*}
          \label{norm}}
  \end{enumerate}
\end{definition}
\begin{remark}
  Note that a linear operator $P$ which satisfies the conditions \ref{fixpoint} and \ref{norm}
  in \Cref{def:QuasiProjection} always gives
  \begin{equation*}
    \mnorm{f-Pf}_{\bfun(\Omega)}
    \leq (1+\tau)E_J(f;\CalB)_{\bfun(\Omega)}.
  \end{equation*}
  This allows us to replace the quantity $(1+\tau)E_J(\CalF;\CalB)_{\bfun(\Omega)}$
  in several statements below simply by the operator norm $\mnorm{\id-P}_{\CalF \to  \bfun(\Omega)}$,
  with $\id$ denoting the embedding $\CalF\hookrightarrow\bfun(\Omega)$.
  Also, we would like to mention that the dependence of quasi-projections of the choice of the dictionary $\CalB$ is visible through the appearance of $\Sigma_n$, $\Sigma_{\kappa n}$, $V_J$, and $V_{J^\ast}$ in \Cref{def:QuasiProjection} but is omitted in the notation. 
\end{remark}

\subsection{Gelfand, Kolmogorov and approximation numbers}
\label{sect:widths}
For benchmark reasons we define the following asymptotic quantities
for a given quasi-normed space $\CalF$ compactly embedded into another normed space $X$,
e.g., $X = \bfun(\Omega)$ or $X = L^2(\mu)$.
Detailed expositions on these quantities can be found in the books \cite{Pietsch1987,Pinkus1985}.

\begin{definition}[Kolmogorov numbers]
  For a quasi-normed space $\CalF$ compactly embedded into $X$ we define the $n$th \emph{Kolmogorov number} of $\CalF$ as
  \begin{equation*}
    d_n(\CalF)_X
    \defeq \inf_{\dim(L) \leq n}
             \sup_{f\in B_\CalF}
               \inf_{g\in L}
                 \mnorm{f-g}_X
  \end{equation*}
 where the outermost infimum runs over the subspaces $L$ of $X$ of dimension at most $n$.
\end{definition}

\begin{definition}[Gelfand numbers]
  For a quasi-normed space $\CalF$ compactly embedded into $X$ we define the $n$th \emph{Gelfand number} of $\CalF$ as
  \begin{equation*}
    c_n(\CalF)_X
    \defeq \inf_{\codim(L) \leq n}
             \sup_{\substack{f \in L\cap B_{\CalF}}}
               \mnorm{f}_X
  \end{equation*}
  where the outermost infimum runs over the closed subspaces $L$ of $\CalF$ of codimension at most $n$.
\end{definition}
We refer to \cite[Remark~2.3]{DirksenUl2018} for a discussion of similar quantities.
Finally, we recall the definition of the approximation numbers, which we need in a more general context, namely for linear operators $T:\CalF \to X$. 

\begin{definition}[approximation numbers]
  For a quasi-normed space $\CalF$ and a linear operator $T:\CalF \to X$ we define the $n$th \emph{approximation number} of $T$ as
  \begin{equation}\label{eq:approx}
    a_n(T)
    \defeq \inf_{\rank A \leq n}
             \sup_{f\in B_{\CalF}}
               \mnorm{Tf-Af}_X,
\end{equation}
where the outermost infimum runs over the bounded linear operators $A:\CalF\to X$ whose rank is at most $n$.
If $\CalF$ is continuously embedded into $X$ and\linebreak $T = \id:\CalF \to X$, we write $a_n(\CalF)_X$ for $a_n(T)$.
\end{definition}

\begin{remark}\label{rem:widths}
The above definitions coincide with the classical approach to $s$-numbers for operators, see \cite[Chapter~2]{Pietsch1987}.
We mostly restrict to the case of the identity/embedding operator from $\CalF$ to $X$, assuming the continuous embedding.
As pointed out by Heinrich \cite{Heinrich1989}, there are certain issues when comparing for instance approximation numbers and linear widths, which are usually defined for a bounded set $K \subset X$ as
\begin{equation}\label{eq:widths}
	\lambda_n(K)_X \defeq \inf_{\substack{A:X\to X\\
	\rank A \leq n}}\sup_{x\in K}\mnorm{x-Ax}_X\,.
\end{equation}
However, in case of $X=H$ being a Hilbert space we always have 
\begin{equation*}
d_n(\CalF)_H = a_n(\CalF)_H = \lambda_n(B_\CalF)_H.
\end{equation*} 
Moreover, for any continuous linear operator $T:\CalF \to H$ it also holds
\begin{equation*}
	a_n(T:\CalF \to H) = \lambda_n(T(B_\CalF))_H
\end{equation*}
even in case of quasi-Banach spaces $\CalF$, see \Cref{thm:approximation-numbers-linear-widths} for a proof.
\end{remark}

We also have 
\begin{equation*}
	a_n(\CalF)_X \leq \varrho^{\mathrm{lin}}_n (\CalF)_X
\end{equation*}
and 
\begin{equation*}
	 c_n(\CalF)_X \leq \varrho_n(\CalF)_X
\end{equation*}
when $\CalF$ is a Banach space of functions $\Omega \to\CC$ which is continuously embedded
into $\bfun(\Omega)\hookrightarrow X$, see \Cref{lem:gelfand-sampling}.

\subsection{Sampling numbers and basis pursuit denoising }
\label{sect:snumbers}

In the definition of the sampling numbers $\varrho_n(\CalF)_X$, there appears an infimum over
(possibly nonlinear) \emph{reconstruction maps} $R:\CC^m\to X$.
Upper bounds on the sampling numbers can be established by investigating one particular reconstruction map.
Our choice uses the optimization problem called \emph{basis pursuit denoising} as a building block,
see \cite[Chapter~15]{RauhutCompressiveSensingBook} for a discussion of iterative algorithms for its numerical solution.

\begin{definition}\label{def:reconstruction-operator}
  Let $\eta>0$, $J^\ast\subset I$ a finite set, $t_1,\ldots,t_m\in\Omega$.
  Put
  \begin{equation*}
      A \defeq \lr{b_j (t_\ell)}_{\ell \in \FirstN{m}, j \in J^\ast}
      \in \CC^{m\times J^\ast}
  \end{equation*}
  and for each $y \in \CC^m$,
  \begin{equation}
    R_\eta(y)
    \defeq \sum_{j \in J^\ast} (x^\#(y))_j  b_j
    \in V_{J^\ast}\subset \bfun(\Omega)
    ,
    \label{eq:ReconstructionOperator}
  \end{equation}
  where $x^\# (y) \in \CC^{J^\ast}$ is any (fixed) solution of the minimization problem
  \begin{equation}
    \inf_{z \in \CC^{J^\ast}} \mnorm{z}_{\ell^1(J^\ast)}
    \quad \text{subject to} \quad
    \mnorm{ A z - y }_{\ell^2(m)} \leq \eta \sqrt{m}
    \label{eq:MinimizationProblem}
  \end{equation}
  if such a solution exists, and $x^\#(y)=0$ otherwise.
  This defines a (not necessarily linear) function $R_\eta:\CC^m \to \bfun(\Omega)$.
\end{definition}
Note that the reconstruction method defined in \Cref{def:reconstruction-operator} is oblivious to any function classes $\CalF$.
Still it requires the choice of certain parameters such as the number $\eta>0$.
In \Cref{thm:restrictedbestnterm}, we will show that for a certain choice of $\eta$ dependent on $\CalF$, the worst-case $L^2$-approximation error achieved by the reconstruction method $R_\eta$ acting on samples functions from the unit ball of $\CalF$ is (up to a multiplicative constant) bounded above by the following approximation characteristic.

\begin{definition}\label{def:restricted-best-n-term}
Under \Cref{assu:StandingAssumptions}, let $n\in\N$ and $J^\ast\subset I$.
We define the \emph{$J^\ast$-restricted best $n$-term approximation error} of $f\in \bfun(\Omega)$ as
\begin{equation*}
\sigma_{n,J^\ast}(f;\CalB)_{\bfun(\Omega)}\defeq \inf_{g\in\Sigma_n\cap V_{J^\ast}}\mnorm{f-g}_{\bfun(\Omega)}.
\end{equation*}
For a quasi-normed space $\CalF\hookrightarrow \bfun(\Omega)$, we furthermore define
\begin{equation*}
\sigma_{n,J^\ast}(\CalF;\CalB)_X\defeq\sup_{f\in B_\CalF}\sigma_{n,J^\ast}(f;\CalB)_{\bfun(\Omega)}.
\end{equation*}
\end{definition}
 
Given a sparse approximation $f^\ast\in \Sigma_n\cap V_{J^\ast}$ for the unknown function $f$,
the samples of $f$ can be interpreted as noisy samples of $f^\ast$.
Under certain assumptions, results from compressed sensing provide guarantees for the recovery
of sparse signals from noisy samples.
The reconstruction error for $f^\ast$ then gives an estimate for the reconstruction error for $f$.
This leads us to the following theorem.

During the revision of the paper at hand, we were made aware of the article \cite{RauhutWa2016} in which some of the ideas just outlined also appear.
However, the approach in \cite[Theorem~6.1]{RauhutWa2016} suits having the Wiener algebra as the target space in \Cref{def:restricted-best-n-term}, whereas the target space $\bfun(\Omega)$ allows us to employ more easily the literature for upper bounds on best $n$-term approximation errors, see \Cref{sec:SpecificApplications}.
\begin{theorem}\label{thm:restrictedbestnterm}
Under \Cref{assu:StandingAssumptions}, there exist universal, positive constants
$C,\widetilde{C},\gamma$ such that the following holds:
Let $n\in\N$ and let $J^\ast\subset I$ be finite with $N\defeq\card{J^\ast}$.
Let $\CalF\hookrightarrow \bfun(\Omega)$ be a quasi-normed space and put
\begin{align*}
\eta&\defeq \sigma_{n,J^\ast}(\CalF;\CalB)_{\bfun(\Omega)},\\
m&\defeq \ceil{C\cdot K^2\cdot n\cdot \Log{n}^3\cdot\Log{N}}.
\end{align*}
Let $t_1,\ldots,t_m\stackrel{iid}{\sim}\mu$.
Then with probability at least $1-N^{-\gamma\Log{n}^3}$, it holds that
\begin{equation}
\sup_{f\in B_\CalF}\mnorm{f-R_\eta(f(t_1),\ldots,f(t_m))}_{L^2(\mu)}\leq\widetilde{C}\eta,\label{eq:target}
\end{equation}
In addition, the approximant $R_\eta(f(t_1),\ldots,f(t_m))$ is contained in $V_{J^\ast}$.
\end{theorem}

The proof of the above theorem crucially relies on the following result
from compressive sensing.

\begin{theorem}[{\cite[Theorems 4.2 and 4.3]{RauhutWard2012}}]\label{thm:CSResult}
  There exist universal constants $C, C_1, C_2$, $\gamma>0$ such that the following holds:
  Let $N\geq 2$ and $(\phi_j)_{j \in \FirstN{N}} \subset L^2(\mu)$ be an orthonormal system with
  $\max_{j\in \FirstN{N}} \mnorm{\phi_j}_{\bfun(\Omega)} \leq K$ and let $s,m \in \N$ satisfy
  \begin{equation*}
      m \geq C \cdot K^2 \cdot s\cdot \Log{s}^3\cdot \Log{N}.
  \end{equation*}
  Put $A = (\phi_j(t_\ell))_{\ell \in \FirstN{m}, j \in \FirstN{N}}$
  for $t_1,\ldots,t_m \overset{iid}{\sim} \mu$.
  Then with probability at least $1 - N^{-\gamma\Log{s}^3}$
  with respect to the choice of $t_1,\ldots,t_m$ the following holds:
   Given any $\eta > 0$, $x\in \CC^N$, $y\in \CC^m$ with $y = A x + e$ and $\mnorm{e}_{\ell^2(m)} \leq \eta \sqrt{m}$,
   and a solution $x^\# \in \CC^N$ of the minimization problem
  \begin{equation*}
    \inf_{z \in \CC^N}
      \mnorm{z}_{\ell^1(N)}
    \quad \text{subject to} \quad
    \mnorm{ A z - y}_{\ell^2(m)} \leq \eta \sqrt{m}
    ,
  \end{equation*}
  then
  \begin{equation*}
    \mnorm{x - x^\# }_{\ell^2(N)}
    \leq \frac{C_1}{\sqrt{s}} \sigma_s(x)_1 + C_2 \eta
    ,
  \end{equation*}
  where
  \begin{equation*}
    \sigma_s(x)_1
    \defeq \inf_{z \in \CC^N, \mnorm{z}_{\ell^0(N)} \leq s} \mnorm{ x - z}_{\ell^1(N)},
  \end{equation*}
  with $\mnorm{z}_{\ell^0(N)} \defeq \card{\setcond{ j \in \FirstN{N} }{ z_j \neq 0}}$.
\end{theorem}

Given the bound in \Cref{thm:CSResult}, we can now prove \Cref{thm:restrictedbestnterm}.

\begin{proof}[{Proof of \Cref{thm:restrictedbestnterm}}]
Let $(\phi_j)_{j\in\FirstN{N}}$ be an enumeration of $(b_j)_{j\in J^\ast}$.
Note that (up to identification of the indices) the matrix $A$
from the statement of \Cref{thm:restrictedbestnterm} (or \Cref{def:reconstruction-operator}) coincides with the matrix $A$ from \Cref{thm:CSResult}.
Thus, we know that with probability at least $1-N^{-\gamma\Log{n}^3}$ with respect to
the choice of $t_1,\ldots,t_m$, the conclusion of \Cref{thm:CSResult} holds (for $s=n$).
It remains to show that under this condition \eqref{eq:target} holds.
Let $f\in B_\CalF$ be arbitrary.
Since $V_{J^\ast}$ is a finite-dimensional vector space
and since $\Sigma_n\cap V_{J^\ast}\subset V_{J^\ast}$ is closed as a finite union of subspaces,
there exists $f^\ast\in\Sigma_n\cap V_{J^\ast}$ satisfying
$\mnorm{f-f^\ast}_{\bfun(\Omega)}\leq\sigma_{n,J^\ast}(\CalF;\CalB)_{\bfun(\Omega)}=\eta$.
Since $f^\ast\in\Sigma_n\cap V_{J^\ast}$, by choice of $(\phi_j)_{j\in\FirstN{N}}$,
we can write $f^\ast=\sum_{j=1}^N x_j\phi_j$ with $x\in \CC^N$ and $\mnorm{x}_{\ell^0(N)}\leq n=s$.
In particular, this implies $\sigma_s(x)_1=0$.
Since $\mnorm{f-f^\ast}_{\bfun(\Omega)}\leq \eta$, we know for
\begin{equation*}
y\defeq (f(t_1),\ldots,f(t_m))\quad\text{ and }\quad e\defeq y-(f^\ast(t_1),\ldots,f^\ast(t_m))
\end{equation*}
that $\mnorm{e}_{\ell^2(m)}\leq\mnorm{e}_{\ell^\infty(m)}\sqrt{m}\leq\eta\sqrt{m}$.
Also we have $Ax+e=y$ because $(Ax)_\ell=\sum_{j=1}^N x_j \phi_j(t_\ell)=f^\ast(t_\ell)$.
Hence, by choice of $x^\#=x^\#(y)$ and by the bound from \Cref{thm:CSResult}, we see
\begin{equation*}
\mnorm{x-x^\#}_{\ell^2(N)}\leq \frac{C_1}{\sqrt{s}}\sigma_s(x)_1+C_2\eta=C_2\eta
\end{equation*}
and thus (because of $f^\ast=\sum_{j=1}^N x_j\phi_j$ and $R_n(y)=\sum_{j\in J^\ast} (x^\#(y))_jb_j$)
\begin{align*}
\mnorm{f-R_\eta(f(t_1),\ldots,f(t_m))}_{L^2(\mu)}&=\mnorm{f-R_\eta(y)}_{L^2(\mu)}\\
&\leq \mnorm{f-f^\ast}_{L^2(\mu)}+\mnorm{f^\ast-R_\eta(y)}_{L^2(\mu)}\\
&\leq \mnorm{f-f^\ast}_{L^2(\mu)}+\mnorm{x-x^\#}_{\ell^2}\\
&\leq (1+C_2)\eta.
\qedhere
\end{align*}
\end{proof}

The quantity $\sigma_{n,J^\ast}(\CalF;\CalB)_{\bfun(\Omega)}$ can be bounded above
by a combination of $E_J(f;\CalB)_{\bfun(\Omega)}$ and $\sigma_n(f;\CalB)_{\bfun(\Omega)}$ when a quasi-projection is available.
This is made explicit by the following lemma.

\begin{lemma}\label{lem:conn}
Under \Cref{assu:StandingAssumptions}, let $J,J^\ast\subset I$, $\kappa,n\in\N$, $\tau>0$, and $P:\bfun(\Omega)\to \bfun(\Omega)$ with the following properties
\begin{enumerate}[label={(\roman*)},leftmargin=*,align=left,noitemsep]
\item{$P(\Sigma_n)\subset\Sigma_{\kappa n}$,\label{sparse}}
\item{$Pf=f$ for $f\in V_J$,\label{fix}}
\item{$Pf\in V_{J^\ast}$ for all $f\in \bfun(\Omega)$,\label{codomain}}
\item{$P:X\to X$ is linear and bounded with $\mnorm{P}_{\bfun(\Omega)\to \bfun(\Omega)}\leq\tau$.}
\end{enumerate}
Then we have for any $f\in \bfun(\Omega)$ that
\begin{equation*}
\sigma_{\kappa n,J^\ast}(f;\CalB)_{\bfun(\Omega)}\leq (1+\tau)E_J(f;\CalB)_{\bfun(\Omega)}+\tau\sigma_n(f;\CalB)_{\bfun(\Omega)}.
\end{equation*}
\end{lemma}

\begin{proof}
Let $\eps>0$.
Choose $f_n\in\Sigma_n$ such that $\mnorm{f-f_n}_{\bfun(\Omega)}\leq \sigma_n(f;\CalB)_{\bfun(\Omega)}+\eps$.
Choose $g\in V_J$ such that $\mnorm{f-g}_{\bfun(\Omega)}\leq E_J(f;\CalB)_{\bfun(\Omega)}+\eps$.
Note that $g=Pg$ by property \ref{fix}.
Define $f^\ast\defeq Pf_n$ and note by \ref{sparse} and \ref{codomain} that $f^\ast\in \Sigma_{\kappa n}\cap V_{J^\ast}$.
Therefore
\begin{align*}
&\sigma_{\kappa n,J^\ast}(f;\CalB)_{\bfun(\Omega)}\\
&\leq\mnorm{f-f^\ast}_{\bfun(\Omega)}\\
&\leq \mnorm{f-g}_{\bfun(\Omega)}+\mnorm{g-f^\ast}_{\bfun(\Omega)}\\
&\leq E_J(f;\CalB)_{\bfun(\Omega)}+\eps+\mnorm{P(g-f_n)}_{\bfun(\Omega)}\\
&\leq E_J(f;\CalB)_{\bfun(\Omega)}+\eps+\tau\mnorm{g-f_n}_{\bfun(\Omega)}\\
&\leq E_J(f;\CalB)_{\bfun(\Omega)}+\eps+\tau(\mnorm{g-f}_{\bfun(\Omega)}+\mnorm{f-f_n}_{\bfun(\Omega)})\\
&\leq E_J(f;\CalB)_{\bfun(\Omega)}+\eps+\tau(E_J(f;\CalB)_{\bfun(\Omega)}+\sigma_n(f;\CalB)_{\bfun(\Omega)}+2\eps)\\
&=(1+\tau)E_J(f;\CalB)_{\bfun(\Omega)}+\tau\sigma_n(f;\CalB)_{\bfun(\Omega)}+2\tau\eps+\eps.
\end{align*}
Since $\eps>0$ was arbitrary, we are done.
\end{proof}

The following result, the main result of this paper, shows that the optimal sampling error can be (essentially) bounded by the nonlinear approximation error
if one is willing to allow an additional logarithmic factor in the number of sampling points.
This might not seem quite obvious at first due to the occurrence of the operator norm
$\mnorm{P}_{\bfun(\Omega) \to \bfun(\Omega)}$ and the linear approximation error $E_J(\CalF;\CalB)_{\bfun(\Omega)}$.
Yet we will see \Cref{sec:SpecificApplications} that in many concrete settings, the term $\sigma_n(\CalF;\CalB)_{\bfun(\Omega)}$
makes the most significant contribution to the right-hand side.

\begin{theorem}\label{thm:MainAbstractResult}
Under \Cref{assu:StandingAssumptions}, there exist universal, positive constants
$C,\widetilde{C},\gamma$ such that the following holds:
Let $\CalF$ be a quasi-normed function space which compactly embeds into $\bfun(\Omega)$.
Moreover, choose $\kappa,n\in\N$ and finite sets $J,J^\ast\subset I$
such that there exists a $(\kappa,n,J,J^\ast,\tau)$ quasi-projection $P:\bfun(\Omega) \to \bfun(\Omega)$
with $\tau\defeq\mnorm{P}_{\bfun(\Omega)\to\bfun(\Omega)}$.
Put
\begin{equation}\label{eq:eta}
	\eta \defeq  \tau\cdot \sigma_n (\CalF;\CalB)_{\bfun(\Omega)}
                     + (1+\tau)\cdot E_J(\CalF;\CalB)_{\bfun(\Omega)}
\end{equation}
and $N\defeq \card{J^\ast}$.
Drawing at least 
\begin{equation}\label{eq:number_samples}
 m \defeq \ceil{C\cdot K^2\cdot \kappa \cdot \Log{\kappa}^3 \cdot n\cdot \Log{n}^3\cdot \Log{N}}
\end{equation}
sampling points $t_j$ independently
and identically distributed from $\mu$, i.e.,\linebreak $t_1,\ldots,t_m$ $\overset{iid}{\sim} \mu$,
then it holds with probability at least $1 - N^{- \gamma\Log{n}^3}$ that
\begin{equation}
  \sup_{f\in B_\CalF}\mnorm{f - R_\eta(f(t_1),\ldots,f(t_m))}_{L^2(\mu)}\\
  \leq \widetilde{C}\eta.
  \label{eq:ExplicitTargetEstimate}
\end{equation}
In addition, the approximant $R_\eta(f(t_1),\ldots,f(t_m))$ is contained in $V_{J^\ast}$.
\end{theorem}
\begin{proof}
This follows by combining \Cref{thm:restrictedbestnterm} (with $\kappa n$ in place of $n$) and \Cref{lem:conn}.
Also, observe that
\begin{align*}
\log(\kappa n+1)&\leq \log((\kappa+1)(n+1))=\log(\kappa+1)+\log(n+1)\\
&\lesssim \log(\kappa+1)\cdot \log(n+1).\qedhere
\end{align*}
\end{proof}

Note that, since $N \geq 2$, the number $1 - N^{- \gamma\Log{n}^3}$ and therefore
also the probability of choosing \enquote{good} sampling points $t_1,\ldots,t_m$ is close to $1$.
As a corollary of \Cref{thm:MainAbstractResult}, we obtain the corresponding bound on the sampling numbers.

\begin{corollary}\label{Cor:general_sampling}
  Under the assumptions of \Cref{thm:MainAbstractResult}, it holds
  \begin{align*}
   & \varrho_{\ceil{C K^2 \kappa\Log{\kappa}^3  n \Log{n}^3 \Log{N}}}(\CalF)_{L^2(\mu)} \\
    &\leq \widetilde{C}(\mnorm{P}_{\bfun(\Omega)\to\bfun(\Omega)}\cdot \sigma_n (\CalF;\CalB)_{\bfun(\Omega)}\\
    &\qquad + (1 + \mnorm{P}_{\bfun(\Omega)\to\bfun(\Omega)})\cdot E_J(\CalF;\CalB)_{\bfun(\Omega)}).
  \end{align*}
\end{corollary}

\subsection{Improved RIP bounds for the Fourier system}
\Cref{thm:CSResult} comprises two statements: Measurement matrices whose entries are the elements of a bounded orthonormal system evaluated at independently and identically distributed points and scaled by the inverse of the square root of the number of rows have the restricted isometry property with high probability, and measurement matrices with the restricted isometry property allow for robust reconstruction of sparse signals.

In the same spirit, the papers \cite{Bourgain2014,HavivRe2017} show that a matrix obtained by rescaling uniformly at random chosen rows from a large unitary matrix has the restricted isometry property with high probability.
In particular, the combination of \cite[Theorem~4.5]{HavivRe2017} applied to matrix representations of the multivariate discrete Fourier transform (DFT) and \cite[Theorem~6.12]{RauhutCompressiveSensingBook} leads to an improvement of one power of $\Log{n}$ compared to \Cref{thm:CSResult}, in case $(\Omega,\mu)$ is the $d$-torus $\T^d\cong [0,1)^d$ with the normalized Haar measure, and $(b_j)_{j\in I}$ is the multivariate Fourier basis $\CalT^d=(\ee^{2\uppi\ii \skpr{k}{\bullet}})_{k\in\Z^d}$.

Before giving the result, let us describe the construction in \cite[Theorem~4.5]{HavivRe2017} for our setting in more detail.
Let $D\in\N$.
We evaluate the functions $e_k:\T\to\CC$, $e_k(x)\defeq \ee^{2\uppi\ii kx}$ from the univariate Fourier basis $\CalT$ at equidistant points\linebreak $t_\ell\defeq\frac{\ell}{2D+1}$ from the unit interval, for $k\in\setn{-D,\ldots,D}$ and $\ell\in\setn{0,\ldots, 2D}$.
The matrix 
\begin{equation*}
Q_{2D+1}=\frac{1}{\sqrt{2D+1}}(e_k(t_\ell))_{\substack{\ell=0,\ldots,2D\\k=-D,\ldots,D}}
\end{equation*}
is the usual $(2D+1)\times (2D+1)$ DFT matrix after a permutation of the columns, so in particular $Q_{2D+1}$ is a unitary matrix.
The multidimensional DFT is separable over dimensions, so the Kronecker product of univariate DFT matrices will serve as a matrix representation of the multivariate DFT, see \cite[p.~230]{PlonkaPoStTa2018}.
In particular, the $d$-fold Kronecker product $Q=Q_{2D+1}\otimes\ldots\otimes Q_{2D+1}$ is a unitary matrix again, see \cite[Theorem~3.42]{PlonkaPoStTa2018}.
Still, the matrix $Q$ is of the form
\begin{equation*}
Q=(2D+1)^{-\frac{d}{2}}(\phi_j(u_\ell))_{\ell,j\in\FirstN{(2D+1)^d}}
\end{equation*}
where $(\phi_j)_{j\in\FirstN{(2D+1)^d}}$ is an enumeration of the functions $e_k\defeq\ee^{2\uppi\ii\skpr{k}{\bullet}}$\linebreak for $k\in [-D,D]^d\cap\Z^d$ and $(u_\ell)_{\ell\in\FirstN{(2D+1)^d}}$ is an enumeration of the grid points $\frac{1}{2D+1}\setn{0,\ldots,2D}^d$.
The subsampling method by \cite[Theorem~4.5]{HavivRe2017} now constructs a matrix $A$ by choosing $m$ rows from $Q$ uniformly at random and multiplying by $(2D+1)^{\frac{d}{2}}$.
In our case, this amounts to directly setting up a matrix $A = (\phi_j(u_\ell))_{\ell \in \FirstN{m}, j \in \FirstN{(2D+1)^d}}$, where again $(\phi_j)_{j\in\FirstN{(2D+1)^d}}$ is an enumeration of the functions $(e_k)_{k\in [-D,D]^d\cap\Z^d}$ but the points $u_1,\ldots,u_m$ are now drawn i.i.d.\ with respect to the uniform measure on $\frac{1}{2D+1}\setn{0,\ldots,2D}^d$.

\begin{theorem}\label{thm:improved-fourier}
There are universal constants $C,C_1,C_2,\gamma>0$ such that the following holds:
Let $D\in\N$ be sufficiently large and $d,s\in \N$.
For
\begin{equation*}
m\geq C\cdot d\cdot s \cdot\Log{s}^2\cdot\log(2D+1),
\end{equation*}
draw points $u_1,\ldots,u_m$ independently
and identically distributed from the uniform measure on the grid $\frac{1}{2D+1}\setn{0,\ldots,2D}^d$, and set
\begin{equation*}
A = (\phi_j(u_\ell))_{\ell \in \FirstN{m}, j \in \FirstN{(2D+1)^d}}
\end{equation*}
where $(\phi_j)_{j\in\FirstN{(2D+1)^d}}$ is an enumeration of $(e_k)_{k\in [-D,D]^d\cap\Z^d}$.
Then with probability at least $1-(2D+1)^{-\gamma\Log{s}d}$ with respect to the choice of $u_1,\ldots,u_m$, the following holds:
Given any $\eta > 0$, $x \in \CC^{(2D+1)^d}$, $y\in\CC^m$ with $y = A x + e$ and $\mnorm{e}_{\ell^2(m)} \leq \eta\sqrt{m} $, and a solution $x^\# \in \CC^{(2D+1)^d}$ of the minimization problem
\begin{equation*}
\inf_{z \in \CC^{(2D+1)^d}} \mnorm{z}_{\ell^1((2D+1)^d)} \quad \text{subject to} \quad \mnorm{ A z - y }_{\ell^2(m)} \leq \eta\sqrt{m},
\end{equation*}
then
\begin{equation*}
\mnorm{x-x^\#}_{\ell^2((2D+1)^d)}\leq \frac{C_1}{\sqrt{s}}\sigma_s(x)_1+C_2\eta.
\end{equation*}
\end{theorem}

%% file: PolynomialBases.tex
\section{Special cases: polynomial bases}%
\label{sec:FourierBasis}

At first sight, it might appear that the term $\mnorm{P}_{\bfun(\Omega) \to \bfun(\Omega)}$
appearing in \Cref{thm:MainAbstractResult} might imply a suboptimal dependence
on the input dimension $d$ in many cases.
In this section, we point out one setting in which this is not the case.
Namely, we consider the Fourier basis $\CalB = \CalT^d \defeq (\ee^{2 \uppi \ii \skpr{k}{\bullet}})_{k \in \Z^d}$
on the torus $\Omega=\T^d\defeq \R^d/\Z^d$ together with the normalized Haar measure $\mu$ on $\T^d$.
In addition, we consider some univariate orthogonal polynomial systems $\CalB=(\gp_n)_{n\in\N_0}$
on $\Omega=\pmone$ together with some (probability) measure $\mu$.

\subsection{Trigonometric polynomials}
\label{sub:FourierBasis1D}

Functions on the torus $\T^d$ can be identified with functions on $\R^d$ which are $1$-periodic
in each coordinate direction, or just functions in $[0,1)^d$.
Whenever we write $L^p(\T^d)$ in this subsection, we mean the $L^p$ space
with respect to the normalized Haar measure $\mu$ on $\T^d$.
If we denote by $[x] \in \R^d / \Z^d$ the equivalence class of $x \in \R^d$,
then $\int_{\T^d}f(x)\dd\mu(x) = \int_{[0,1)^d} f([x]) \dd x$.

We first consider the case $d = 1$ and recall several facts
about the family of de la Vallée Poussin operators and kernels,
taken from \cite[Section~2]{FilbirDeLaValleePoussinKernels}.
For $n,m \in \N_0$ with $m \leq n$ and $k \in \Z$, define
\begin{equation}
  a_k^{(n,m)}
  \defeq \begin{cases}
       1                              & \text{if } \abs{k} \leq n - m, \\
       \frac{n + m + 1 - \abs{k}}{2 m + 1} & \text{if } n - m + 1 \leq \abs{k} \leq n + m, \\
       0                              & \text{otherwise}.
     \end{cases}
  \label{eq:DeLaValleePoussinCoefficients}
\end{equation}
Furthermore, define the \emph{de la Vallée Poussin operator} $P^{(n,m)}$ as
\begin{equation*}
  P^{(n,m)} f
  \defeq \sum_{k \in \Z} a_k^{(n,m)} \widehat{f} (k) e_k
  = f \ast K^{(n,m)},
\end{equation*}
where $K^{(n,m)}\defeq \sum_{k \in \Z} a_k^{(n,m)}e_k$ with $e_k(x) \defeq \ee^{2\uppi \ii kx}$
is the \emph{de la Vallée Poussin kernel} and $\widehat{f}(k) = \skpr{f}{e_k}_{L^2(\T)}$.
Directly from the definition and from \Cref{eq:DeLaValleePoussinCoefficients},
it is then easy to see for $\CalB=\CalT^d$ that
\begin{enumerate}[label={(\roman*)},leftmargin=*,align=left,noitemsep]
  \item $P^{(n,m)}(\Sigma_\ell) \subset \Sigma_\ell$ for all $\ell \in \N$,

  \item $P^{(n,m)} f = f$ for all $f \in V_J$ for $J \defeq J_{n,m} \defeq \setcond{ k \in \Z }{ \abs{k} \leq n - m}$,

  \item $P^{(n,m)} f \in V_{J^\ast}$ for all $f \in L^2(\T)$
        for $J^\ast \defeq J_{n,m}^\ast \defeq \setcond{ k \in \Z }{ \abs{k} \leq n + m}$,

  \item $P^{(n,m)} : \bfun(\T) \to \bfun(\T)$ is well-defined and bounded.
\end{enumerate}
In other words, $P^{(n,m)}$ is a $(1,\ell,J,J^\ast,\tau)$ quasi-projection
for any $\ell \in \N$, $J,J^\ast$ as above, and any $\tau\geq \mnorm{P^{(n,m)}}_{\bfun(\T) \to \bfun(\T)}$.
It remains to estimate the operator norm $\mnorm{P^{(n,m)}}_{\bfun(\T) \to \bfun(\T)}$.
To this end, note that by Young's convolution inequality (see, e.g, \cite[Proposition~(2.39)]{FollandAHAFirstEdition})
we have the estimate $\mnorm{ P^{(n,m)}}_{\bfun(\T) \to \bfun(\T)} \leq \mnorm{ K^{(n,m)}}_{L^1(\T)}$.
Next, \cite[Equation~(3)]{FilbirDeLaValleePoussinKernels} shows with the Fejér kernel
$F_n \defeq \frac{1}{n + 1} \sum_{k=0}^n\sum_{j = - k}^k e_j$ that
\begin{equation*}
  K^{(n,m)}
  = \frac{n + m + 1}{2 m + 1} F_{n + m} - \frac{n - m}{2 m + 1} F_{n - m - 1}.
\end{equation*}
It is well-known (see e.g. \cite[Example~1.2.18]{GrafakosClassicalFourierThirdEdition})
that $F_n \geq 0$ and $\int_{\T} F_n(x) \dd x = 1$.
Since $\abs{a - b} \leq a + b$ for $a,b \geq 0$, this implies
\begin{equation*}
  \mnorm{ P^{(n,m)} }_{\bfun(\T) \to \bfun(\T)}
  \leq \mnorm{ K^{(n,m)} }_{L^1(\T)}
  \leq \frac{n + m + 1}{2 m + 1} + \frac{n - m}{2 m + 1}
  =    \frac{2 n + 1}{2 m +1}
  .
\end{equation*}

Let us now consider the multivariate case. 
For given $d,M \in \N$, consider the operator
\begin{equation*}
  \CalP_{M}^d : \quad
  f \mapsto f \ast \CalK_M^d
  \quad \text{where} \quad
  \CalK_M^d \defeq K^{((d+1)M, d M)} \otimes \ldots \otimes K^{((d+1)M, d M)}
\end{equation*}
is the $d$-fold tensor product of the kernel $K^{((d+1)M, d M)}$ from the univariate case, i.e.,
\begin{equation*}
(f_1\otimes\ldots\otimes f_d)(x_1,\ldots,x_d) = \prod_{j=1}^d f_j(x_j),\quad x=(x_1,\ldots,x_d)\in \R^d.
\end{equation*}
For the sequence $a_{(k_1,\ldots,k_d)} \defeq a_{k_1}^{((d+1)M, d M)} \cdot \ldots \cdot a_{k_d}^{((d+1)M, d M)}$,
it is then not hard to see that in terms of Fourier coefficients, we have
\begin{equation*}
  \CalP_M^d f
  = \sum_{k \in \Z^d} a_k  \widehat{f}(k)  e_k
  ,
\end{equation*}
and this implies that $\CalP_M^d$ has the following properties:
\begin{enumerate}[label={(\roman*)},leftmargin=*,align=left,noitemsep]
  \item{$\CalP_M^d (\Sigma_\ell) \subset \Sigma_\ell$ for arbitrary $\ell \in \N$,}
  \item{$\CalP_M^d f = f$ for all $f \in V_{J_M^d}$, where
        \begin{equation*}
          J_M^d
          \defeq \setcond{ k \in \Z^d }{ \mnorm{k}_{\ell^\infty(\Z^d)} \leq (d+1) M - d M }
          = [-M,M]^d\cap \Z^d,
        \end{equation*}
}
  \item{$\CalP_M^d f \in V_{J_M^{d,\ast}}$ for all $f \in L^2(\T^d)$, where
        \begin{equation*}
          J_M^{d,\ast}
          \defeq \setcond{ k \in \Z^d}{ \mnorm{k}_{\ell^\infty(\Z^d)} \leq (2d + 1) M}
          ,
       \end{equation*}
}
  \item{$\CalP_M^d : \bfun(\T^d) \to \bfun(\T^d)$ is well-defined and bounded. 
	  In fact, we have for any $p \in [1,\infty]$ that
        \begin{align*}
          \mnorm{ \CalP_M^d }_{L^p(\T^d) \to L^p(\T^d)}
          & \leq \mnorm{ \CalK_M^d }_{L^1(\T^d)}
            =    \mnorm{ K^{((d+1)M, d M)} }_{L^1(\T)}^d \\
          & \leq \lr{\frac{2 (d+1) M + 1}{2 d M + 1}}^d
            =    \lr{ 1 + \frac{2 M}{2 d M + 1} }^d \\
          & \leq \lr{ 1 + \frac{1}{d} }^d
            \leq \ee .
        \end{align*}
}
\end{enumerate}
By applying \Cref{thm:MainAbstractResult} to the setting of the Fourier basis,
with $P = \CalP_M^d$, we thus obtain the following result:

\begin{theorem}\label{thm:GeneralFourierBound}
  There exist universal constants $C, \widetilde{C},\gamma > 0$ with the following property:
  For any $d \in \N$, and any quasi-normed space $\CalF \hookrightarrow \bfun(\T^d)$,
  and arbitrary $n,M \in \N$ with $M \geq 3$, we have
  \begin{equation*}
  \begin{split}
     &\varrho_{\ceil{Cd \Log{d} n\Log{n}^3 \log(M)}} (\CalF)_{L^2(\T^d)}\\ 
     &~~~~\leq \widetilde{C} \lr{
                   \sigma_n(\CalF;\CalT^d)_{\bfun(\T^d)}
                   + E_{[-M,M]^d\cap \Z^d} (\CalF;\CalT^d)_{\bfun(\T^d)}
                 }.
	\end{split}  
  \end{equation*}
  In fact, if 
  \begin{align*}
    m &\defeq \ceil{C\cdot d \cdot \Log{d} \cdot n \cdot \Log{n}^3\cdot \log(M)},\\
    N &\defeq (2(2d+1)M+1)^d,
  \end{align*}  
and if the points
  $t_1,\ldots,t_m \overset{iid}{\sim} \mu$ are sampled uniformly,
  then with probability at least $1 - N^{-\gamma\Log{n}^3}$ with respect to the choice of the sampling points,
  the reconstruction operator $R_\eta$ defined in \Cref{eq:ReconstructionOperator,eq:MinimizationProblem,eq:eta}
  satisfies
  \begin{equation*}
	\begin{split}
      &\mnorm{ f - R_\eta(f(t_1),\ldots,f(t_m))}_{L^2(\T^d)}\\
    &~~~\leq \widetilde{C}
         \lr{
                 \sigma_n(\CalF,\CalT^d)_{\bfun(\T^d)}
                 + E_{[-M,M]^d\cap\Z^d}(\CalF; \CalT^d)_{\bfun(\T^d)}               }
	\end{split}  
  \end{equation*}
 for all $f\in B_\CalF$.
\end{theorem}

\begin{proof}
  This follows from the above considerations by an application of \Cref{thm:MainAbstractResult}
  with $K = 1$, $\kappa = 1$, $P = \CalP_{M}^d$, $J = J_M^d$, and $J^\ast = J_M^{d,\ast}$,
  which then implies $N = \card{J^\ast} = (2(2d+1)M+1)^d$ and hence
  \begin{align*}
    \Log{N}
    & \leq d \log\lr{3 (2d+1) M} \\
    & \leq d \lr{\log(12d) + \log(M)} \\
    & \leq 2 d \log(12d)\cdot\log(M)\\
    & \leq 8 d \Log{d}\cdot\log(M).
    \qedhere
  \end{align*}
\end{proof}
From \Cref{thm:improved-fourier} for $D=(2d+1)M$,
we obtain the following refinement of \Cref{thm:GeneralFourierBound}.

\begin{theorem}\label{thm:improvedGeneralFourierBound}
  There exist universal constants $C, \widetilde{C},\gamma > 0$ with the following property:
  For any $d \in \N$, and any quasi-normed space $\CalF \hookrightarrow \bfun(\T^d)$,
  and arbitrary $n,M \in \N$ with $M \geq 3$, we have
  \begin{equation*}
  \begin{split}
     &\varrho_{\ceil{C d\Log{d} n\Log{n}^2\log(M)}} (\CalF)_{L^2(\T^d)}\\ 
     &~~~~\leq \widetilde{C} \lr{
                   \sigma_n(\CalF;\CalT^d)_{\bfun(\T^d)}
                   + E_{[-M,M]^d\cap \Z^d} (\CalF;\CalT^d)_{\bfun(\T^d)}
                 }.
	\end{split}  
  \end{equation*}
  In fact, if
  \begin{align*}
  m &\defeq \ceil{C\cdot d\cdot\Log{d}\cdot n \cdot\Log{n}^2\cdot\log(M)},\\
  N &\defeq (2(2d+1)M+1)^d,
  \end{align*}
  and if the matrix $A$ is constructed as in \Cref{thm:improved-fourier},
  then with probability at least $1-N^{-\gamma\Log{n}}$ with respect to the choice
  of the points $u_1,\ldots,u_m$ in \Cref{thm:improved-fourier},
  the reconstruction operator $R_\eta$ defined in \Cref{eq:ReconstructionOperator,eq:MinimizationProblem,eq:eta}
  satisfies
  \begin{equation*}
	\begin{split}
      &\mnorm{ f - R_\eta(f(t_1),\ldots,f(t_m))}_{L^2(\T^d)}\\
    &~~~\leq \widetilde{C}
         \lr{
                 \sigma_n(\CalF,\CalT^d)_{\bfun(\T^d)}
                 + E_{[-M,M]^d\cap\Z^d}(\CalF; \CalT^d)_{\bfun(\T^d)}               }
	\end{split}  
  \end{equation*}
for all $f\in B_\CalF$.
\end{theorem}

\subsection{Algebraic polynomials}
\label{chap:algebraic-polynomials}

In this subsection we investigate the case of orthogonal polynomial systems on $\Omega=\pmone$.
Let $\mu$ be some probability measure on $\Omega$. 
By $\CalB=(\gp_n)_{n\in\N_0}$ we denote the unique family of polynomials for which $\deg(\gp_n) = n$,
the leading coefficient of $\gp_n$ is positive, and
\begin{equation*}
  \int_\Omega \gp_m(x)\gp_n(x)\dd\mu(x)
  = \begin{cases}
      1 & \text{if }n=m, \\
      0 & \text{else},
    \end{cases}
\end{equation*}
for all $n,m\in\N_0$, see \cite[Section~2.2]{Szego1975}.
Note that the $\bfun(\pmone)$-norm is nonetheless the \enquote{pure} uniform norm without any weight.

Here we are interested in two particular families of Jacobi polynomials,
namely the orthonormalized \emph{Chebyshev polynomials of the first kind} ($\alpha=-\frac{1}{2}$)
and the orthonormalized \emph{Legendre polynomials} ($\alpha=0$).
For these systems of polynomials, the measure $\mu = \mu_\alpha$
takes the form $\dd\mu_\alpha(x) = c_\alpha(1-x^2)^\alpha\dd x$
where $c_\alpha=\lr{\int_{-1}^1(1-x^2)^\alpha\dd x}^{-1}$.
The inner product of $L^2(\mu_\alpha)$ is then given by
  \begin{equation*}
    \skpr{f}{g}_{L^2(\mu_\alpha)}
    = \int_{-1}^1 f(x)\overline{g(x)}\dd\mu_\alpha(x)
    = \int_{-1}^1 f(x)\overline{g(x)}v_\alpha(x)\dd x.
\end{equation*}
The occuring weight function is $v_\alpha:\pmone\to\R$,
$v_\alpha(x)\defeq c_\alpha(1-x^2)^\alpha$ and the resulting orthonormal polynomial system shall be denoted by
by $\CalB_\alpha=(\gp_n^\alpha)_{n\in\N_0}$.
Note that the factor $c_\alpha$ in the weight function is usually omitted in the literature
on orthogonal polynomials \cite[Chapter~4]{Lebedev1972} but in the present paper,
we restrict ourselves to probability measures, see \Cref{assu:StandingAssumptions}.

Similarly to \cite[p.~97]{Lebedev1972}, $\gp^{-\frac{1}{2}}_0(x)=1$
and $\gp^{-\frac{1}{2}}_n(x)=\sqrt{2}\cos(n \arccos(x))$ for $x\in \pmone$ and $n\in\N$.
Thus the Chebyshev polynomials $\gp^{-\frac{1}{2}}_n$ fit in our framework of \Cref{assu:StandingAssumptions} as
\begin{equation*}
\mnorm{\gp^{-\frac{1}{2}}_n}_{\bfun(\pmone)}\leq\sqrt{2}
\end{equation*}
for all $n\in\N_0$.
However, the Legendre polynomials $\gp^0_n$ satisfy
\begin{equation*}
\mnorm{\gp^0_n}_{\bfun(\pmone)}=(2n+1)^{\frac{1}{2}},
\end{equation*}
see \cite[(4.4.1) and (4.5.2)]{Lebedev1972}.
Still, we will be able to formulate a version of \Cref{thm:MainAbstractResult}
for both Chebyshev and Legendre polynomials in \Cref{thm:multivariate_algebraic}.
The lack of $\bfun$-boundedness of the Legendre polynomials uniformly in $n\in\N$
can be overcome using a change-of-measure technique from \cite{RauhutWard2012},
see \Cref{sect:unbounded}.

Filbir and Themistoclakis \cite[Section~3]{FilbirDeLaValleePoussinKernels} define analogs
of the trigonometric de la Vallée Poussin operators for the algebraic setting.
Likewise, they obtain operators
\begin{equation}\label{eq:DLP}
  P^{(n,m)} f(x)
  \defeq \sum_{k=0}^{n+m}a_k^{(n,m)}\widehat{f}(k)\gp_k(x),	
\end{equation}
for some specific choice of the coefficients $a_k^{(n,m)}$.
Note that unlike in our setting, Filbir and Themistoclakis \cite{FilbirDeLaValleePoussinKernels}
require the polynomials to take the value $1$ at $1$.
However, our different scaling does not affect the overall structure
of the construction of $P^{(n,m)}$ but only changes the coefficients $a_k^{(n,m)}$.
In particular, using \eqref{eq:DLP} and \cite[Proposition~3.2]{FilbirDeLaValleePoussinKernels},
the following observations remain unchanged:
\begin{enumerate}[label={(\roman*)},leftmargin=*,align=left,noitemsep]
\item{$P^{(n,m)}(\Sigma_\ell) \subset \Sigma_\ell$ for all $\ell \in \N$,}
\item{$P^{(n,m)} f = f$ for all $f \in V_J$ where $J \defeq J_{n,m} \defeq \setn{0,\ldots,n-m}$,}
\item{$P^{(n,m)} f \in V_{J^\ast}$ for all $f \in L^2(\mu)$ where 
\begin{equation*}
J^\ast \defeq J_{n,m}^\ast \defeq \setn{0,\ldots,n+m}.
\end{equation*}}
\item{Under some additional assumptions on the orthogonal polynomial system $\CalB$,
    the operator $P^{(n,m)} : \bfun(\pmone) \to \bfun(\pmone)$ is well-defined and bounded,
    even uniformly in $n$ and $m$.
    The uniform bound $\tau$ can be given in terms of the Haar weights of $\CalB$,
  see \cite[Theorem~3.4]{FilbirDeLaValleePoussinKernels}.}
\end{enumerate}
In other words, $P^{(n,m)}$ is a $(1,\ell,J,J^\ast,\tau)$ quasi-projection for any $\ell \in \N$ and for $J,J^\ast,\tau$ as above.

For Chebyshev polynomials and Legendre polynomials, explicit bounds on the operator norm of $P^{(n,m)}$ are given
in \cite[p.~309]{FilbirDeLaValleePoussinKernels}:
\begin{enumerate}[label={(\roman*)},leftmargin=*,align=left,noitemsep]
\item{For $\alpha=-\frac{1}{2}$ (Chebyshev polynomials), we have
      \begin{equation*}
      \mnorm{P^{(n,m)}}_{\bfun(\pmone)\to\bfun(\pmone)}
      \leq\frac{2n+1}{2m+1}.
      \end{equation*}}

\item{For $\alpha=0$ (Legendre polynomials), we have
      \begin{equation*}
        \mnorm{P^{(n,m)}}_{\bfun(\pmone)\to\bfun(\pmone)}
        \leq\frac{(n+1)^2}{(m+1)^2}.
      \end{equation*}}
\end{enumerate}

For specializing \Cref{thm:MainAbstractResult} to Chebyshev and Legendre polynomials,
we aim for $(1,\ell,J,J^\ast,\tau)$ quasi-projections $P$ with 
$J=\setn{0,\ldots,M}$ where $M\in\N$ is chosen in advance.
This is achieved by taking the linear operator $P^{(2M,M)}$ for $P$,
and the operator norms $\mnorm{P^{(2M,M)}}_{\bfun(\pmone)\to \bfun(\pmone)}$
are uniformly bounded above by $\tau=2$ in the case of Chebyshev polynomials,
and by $\tau=4$ in the case of Legendre polynomials. 

\begin{theorem}\label{thm:multivariate_algebraic}
  Let $\CalF$ be a quasi-normed space with $\CalF \hookrightarrow \bfun(\pmone)$.
  Let further $\CalB=\CalB_\alpha$ for $\alpha\in\setn{-\frac{1}{2},0}$.
  Then there are absolute constants $C,\widetilde{C}>0$ such that for arbitrary $n,M \in \N$ with $M \geq 3$, we have
  \begin{equation*}
  \begin{split}
     &\varrho_{\ceil{C n \Log{n}^3\log(M)}} (\CalF)_{L^2(\mu_\alpha)}\\ 
     &~~~~\leq \widetilde{C} \lr{
                   \sigma_n(\CalF;\CalB_\alpha)_{\bfun(\pmone)}
                   + E_{\setn{0,\ldots,M}} (\CalF;\CalB_\alpha)_{\bfun(\pmone)}
                 }.
	\end{split}  
  \end{equation*}
\end{theorem}

\Cref{thm:multivariate_algebraic} is a direct consequence of \Cref{thm:MainAbstractResult}
together with the above discussion of the de la Vallée Poussin quasi-projections
in case of Chebyshev polynomials of the first kind, i.e., $\alpha=-\frac{1}{2}$.
For Legendre polynomials, i.e., $\alpha=0$, \Cref{thm:MainAbstractResult}
is not directly applicable since those systems are not bounded in $\bfun(\pmone)$
and so \Cref{assu:StandingAssumptions} is not fulfilled.
In these cases, we utilize the preconditioning strategy introduced by Rauhut and Ward \cite[Theorem~2.1]{RauhutWard2012}.
The proof is then a direct consequence of \Cref{thm:weighted-legendre} together with \Cref{lem:conn}
for $J \defeq \setn{0,\ldots,M}$ and $J^\ast \defeq\setn{0,\ldots,2M}$.
For details we refer to \Cref{sect:unbounded}.

%% file: Applications.tex
\section{Specific applications}%
\label{sec:SpecificApplications}

In \Cref{sec:FourierBasis}, we provided bounds on the operator norm of the quasi-projection
appearing in \Cref{thm:MainAbstractResult}, when $\CalB$ is the (multivariate) Fourier basis
or an orthogonal polynomial system.
In this section, we elaborate on corresponding function spaces $\CalF$
in cases where bounds on the quantities $\sigma_n(\CalF;\CalB)_X$ and $E_J(\CalF;\CalB)_X$ are available.
In particular, all bounds on the sampling numbers that appear in this section are incarnations of \Cref{thm:improvedGeneralFourierBound} or \Cref{thm:multivariate_algebraic}.

\subsection{Weighted Wiener spaces}

\subsubsection{Mixed Wiener spaces on the torus}

We start with the periodic Wiener classes $\CalA^r_{\mathrm{mix}}(\T^d)$
and $\CalA^r_p(\T^d)$ for $r>0$ in place of $\CalF$ in \Cref{thm:GeneralFourierBound}.
While in the present subsection, the periodic Sobolev spaces with bounded mixed derivates
are just a tool, they and related function spaces will be studied in \Cref{chap:periodic-sobolev} in more detail.

Later in the present subsection, we will also deal with non-periodic counterparts $\CalA^r_{\alpha,p}(\pmone)$
of the Wiener classes in the framework of orthogonal polynomials. 

\begin{definition}\label{def:wiener-sobolev}
  For $r > 0$, let
  \begin{equation*}
  \CalA^r_{\mathrm{mix}}(\T^d)
  \defeq \setcond{f:\T^d\to\CC}
                 {
                   \sum_{k\in\Z^d}
                     \abs{\widehat{f}(k)}
                     \prod_{j=1}^d
                       (1+\abs{k_j})^r
                   <\infty
                 },
  \end{equation*}
  and for $r >\frac{1}{2}$, let
  \begin{equation*}
    H^r_{\mathrm{mix}}(\T^d)
    \defeq \setcond{f:\T^d\to\CC}
                   {
                     \sum_{k\in\Z^d}
                       \abs{\widehat{f}(k)}^2
                       \prod_{j=1}^d
                         (1+\abs{k_j})^{2r}
                     <\infty
                   }
  \end{equation*}
  denote the \emph{weighted mixed Wiener space} and the
  \emph{$L^2$-Sobolev space of periodic functions with bounded mixed derivatives}, respectively.
  The quasi-norms in these spaces are given by
  \begin{align*}
  \mnorm{f}_{\CalA^r_{\mathrm{mix}}(\T^d)}&\defeq \sum_{k\in\Z^d} \abs{\widehat{f}(k)} \prod_{j=1}^d (1+\abs{k_j})^r\\
  \intertext{and}
  \mnorm{f}_{H^r_{\mathrm{mix}}(\T^d)}&\defeq\lr{\sum_{k\in\Z^d}\abs{\widehat{f}(k)}^2\prod_{j=1}^d (1+\abs{k_j})^{2r}}^{\frac{1}{2}},
  \end{align*}
  respectively.
\end{definition}
Both the Banach spaces $\CalA^r_{\mathrm{mix}}(\T^d)$ for $r>0$ and the Hilbert spaces $H^r_{\mathrm{mix}}(\T^d)$ for $r>\frac{1}{2}$ embed compactly into $\bfun(\T^d)$.
Asymptotic estimates for the Gelfand and Kolmogorov numbers of $\CalA^r_{\mathrm{mix}}(\T^d)$
in $L^2(\T^d)$ are collected for later use in the next lemma.
Part~\ref{kolmogorov-nns} is taken from \cite[Theorem~4.6(ii)]{NguyenNgSi2022},
and details for \ref{gelfand-moeller} will be given in \cite{Moeller2023}.
\begin{lemma}\label{Lem:Kolm}
  Let $r>0$.
  For all $n\in\N$, we have
  \begin{enumerate}[label={(\roman*)},leftmargin=*,align=left,noitemsep]
    \item{$c_n(\CalA^r_{\mathrm{mix}}(\T^d))_{L^2(\T^d)}\asymp n^{-r-\frac{1}{2}}\Log{n}^{(d-1)r}$,\label{gelfand-moeller}}

    \item{$a_n(\CalA^r_{\mathrm{mix}}(\T^d))_{L^2(\T^d)} = d_n(\CalA^r_{\mathrm{mix}}(\T^d))_{L^2(\T^d)}\asymp n^{-r}\Log{n}^{(d-1)r}$.\label{kolmogorov-nns}}
  \end{enumerate}
\end{lemma}

Next we prove estimates for the worst-case best $n$-term and linear approximation errors
appearing in \Cref{thm:GeneralFourierBound} for $\CalF=\CalA^r_{\mathrm{mix}}(\T^d)$.
Since there is a $d$-dependent number $\lambda$ such that the $\lambda$-ball $\setcond{f\in\CalA^r_{\mathrm{mix}}(\T^d)}{\mnorm{f}_{\CalA^r_{\mathrm{mix}}(\T^d)}\leq \lambda}$ is a subset of the set $W^{r,0}_A$ studied in \cite{Temlyakov2015}, 
Part~\ref{bestmtermwiener} of the following lemma is in fact a consequence of \cite[Lemma~2.1]{Temlyakov2015}.
Even so, we provide an independent proof.
\begin{lemma}\label{lem:bestmterm}
  \begin{enumerate}[label={(\roman*)},leftmargin=*,align=left,noitemsep]
    \item{\label{bestmtermwiener}
          Let $r>\frac{1}{2}$. For all $n\in\N$, we have
          \begin{equation*}
            \sigma_n(\CalA^r_{\mathrm{mix}}(\T^d);\CalT^d)_{\bfun(\T^d)}
            \lesssim n^{-r-\frac{1}{2}}\Log{n}^{(d-1)r+\frac{1}{2}}.
          \end{equation*}}
    \item{\label{linearwiener}
          Let $r>0$.
          For all $M\in\N$, we have 
          \begin{equation*}
            E_{[-M,M]^d\cap\Z^d}(\CalA^r_{\mathrm{mix}}(\T^d);\CalT^d)_{\bfun(\T^d)}
            \leq M^{-r}.
          \end{equation*}}
  \end{enumerate}
\end{lemma} 

\begin{proof}
For \ref{bestmtermwiener} we use the following commutative diagram
\tikzcdset{arrow style=tikz, diagrams={>=stealth}}
\begin{equation*}
\begin{tikzcd}
\CalA^r_{\mathrm{mix}}(\T^d) \ar[hookrightarrow]{rr}{r>0} \ar[hookrightarrow]{rd}{} && \bfun(\T^d) \\
& H^r_{\mathrm{mix}}(\T^d) \ar[swap,hookrightarrow]{ru}{r>\frac{1}{2}} &              
\end{tikzcd}
\end{equation*}
in which all arrows depict embeddings when the indicated constraints on $r$ are met.
This situation is a specific instance of the diagram right before \Cref{lem:factor} with $\CalF = \CalA^r_{\mathrm{mix}}(\T^d)$, $\mathcal{H} = H^r_{\mathrm{mix}}(\T^d)$, $X=L^{\infty}(\T^d)$ and the dictionary $\mathcal{B}$ denotes the multivariate trigonometric system $\mathcal{T}^d$.
By \Cref{lem:factor} we therefore obtain
\begin{align*}
  &\sigma_{n_1+n_2}(\CalA^r_{\mathrm{mix}}(\T^d);\CalT^d)_{\bfun(\T^d)}\\
  &~~~\leq \sigma_{n_1}(\CalA^r_{\mathrm{mix}}(\T^d);\CalT^d)_{H^r_{\mathrm{mix}}(\T^d)}
           \sigma_{n_2}(H^r_{\mathrm{mix}}(\T^d);\CalT^d)_{\bfun(\T^d)}.
\end{align*}
Now let $n_1 = n_2 = n$.
Stechkin's lemma \cite[Lemma~7.4.1]{DungTeUl2018} gives
\begin{equation*}
\sigma_n(\CalA^r_{\mathrm{mix}}(\T^d))_{H^r_{\mathrm{mix}}(\T^d)}\leq n^{-\frac{1}{2}}.
\end{equation*}
From \cite[Theorem~2]{Belinskii1989} (see also \cite[Theorem~3]{Romanyuk1995}) we obtain
\begin{equation*}
  \sigma_n(H^r_{\mathrm{mix}}(\T^d);\CalT^d)_{\bfun(\T^d)}
  \lesssim n^{-r}\Log{n}^{(d-1)r+\frac{1}{2}}.
\end{equation*}
Hence
\begin{equation*}
  \sigma_{2n}(\CalA^r_{\mathrm{mix}}(\T^d);\CalT^d)_{\bfun(\T^d)}
  \lesssim n^{-r-\frac{1}{2}}\Log{n}^{(d-1)r+\frac{1}{2}},
\end{equation*}
and this easily implies the claim.

For \ref{linearwiener} we have the following simple estimate.
Consider the $d$-variate partial sum operator $S_M$ with respect to the cube $[-M,M]^d \cap \Z^d$,
defined via
\begin{equation*}
	S_M f
  \defeq \sum_{k \in [-M,M]^d \cap \Z^d}
           \widehat{f}(k)
           \ee^{2\uppi \ii\skpr{ k}{\bullet}}.
\end{equation*}
Consider
\begin{align}
  \mnorm{f-S_Mf}_{\bfun(\T^d)}
  & \leq \sum_{k\in\Z^d\setminus [-M,M]^d}
           \abs{\widehat{f}(k)} \nonumber \\
  & \leq \mnorm{f}_{\CalA^r_{\mathrm{mix}}(\T^d)}
         \sup_{k\in \Z^d\setminus [-M,M]^d}
           \prod_{j=1}^d(1+\abs{k_j})^{-r} \label{eq:SM} \\
  & \leq \mnorm{f}_{\CalA^r_{\mathrm{mix}}(\T^d)}M^{-r}.
  \tag*{$\qed$}
\end{align}%
\renewcommand{\qedsymbol}{}%
\end{proof}

Using \Cref{thm:improvedGeneralFourierBound} together with \Cref{lem:bestmterm} we obtain the following result.

\begin{corollary}\label{thm:sampling-wiener-periodic}
Let $r > \frac{1}{2}$.
Then there are universal constants $C,\widetilde{C} > 0$ and a constant $C_{r,d}>0$ such that 
\begin{align*}
 &\norel \varrho_{\ceil{C d\Log{d}n\Log{n}^2\log(M)}}(\CalA^r_{\mathrm{mix}}(\T^d))_{L^2(\T^d)}\\
&  \leq \widetilde{C}(\sigma_n(\CalA^r_{\mathrm{mix}}(\T^d);\CalT^d)_{\bfun(\T^d)}+M^{-r}),
\end{align*}
and with $M = \Big\lfloor n^{\frac{r+\frac{1}{2}}{r}} \Big\rfloor$, we have
\begin{equation*}
  \varrho_{\ceil{C_{r,d} n\Log{n}^3}}(\CalA^r_{\mathrm{mix}}(\T^d))_{L^2(\T^d)}
  \lesssim n^{-r-\frac{1}{2}}\Log{n}^{(d-1)r+\frac{1}{2}}.
\end{equation*}
\end{corollary}

\begin{remark}\label{rem:literature}
\begin{enumerate}[label={(\roman*)},leftmargin=*,align=left,noitemsep]
  \item{\label{linear-nonlinear}By \Cref{sect:widths} and \Cref{Lem:Kolm} the \emph{linear} sampling numbers satisfy the lower bound
        \begin{equation*}
          \varrho_n^{\mathrm{lin}}(\CalA^r_{\mathrm{mix}}(\T^d))_{L^2(\T^d)}
          \geq a_n(\CalA^r_{\mathrm{mix}}(\T^d))_{L^2(\T^d)}
          \gtrsim n^{-r}\Log{n}^{(d-1)r}.
        \end{equation*}
        However, from \Cref{thm:sampling-wiener-periodic} together with \Cref{lem:technical} in the appendix we obtain for the (nonlinear) sampling numbers the upper bound 
		\begin{equation*}
			\varrho_n(\CalA^r_{\mathrm{mix}}(\T^d))_{L^2(\T^d)} \lesssim n^{-r-\frac{1}{2}}\Log{n}^{3(r+\frac{1}{2})+(d-1)r+\frac{1}{2}}.	
		\end{equation*}   
        So we get a speed-up of $n^{-\frac{1}{2}}$ in the main term compared to \emph{any} linear recovery algorithm, see also \Cref{sect:widths}.
The result in \Cref{thm:sampling-wiener-periodic} can be extended to $r>0$, which is done in \cite{Moeller2023} based on \cite[Theorems~2.5 and~2.6]{Temlyakov2015}.
Note that the main rate is optimal as \Cref{Lem:Kolm} shows, see again \cite{CreutzigWo2004}.
We observe a $d$-independent logarithmical gap compared to the Gelfand numbers.}
  \item{\label{temlyakov}If we compare to Temlyakov's result \cite[Theorem~1.1]{Temlyakov2021}, i.e.,
        \begin{equation*}
          \varrho_{bn}(\CalF)_{L^2(\mu)}\lesssim Cd_n(\CalF)_{\bfun(\Omega)},
        \end{equation*}
        where $b,C>0$ denote universal constants, this also gives a rate of $n^{-r}$ (up to logarithms) for $\CalF = \CalA^r_{\mathrm{mix}}(\T^d)$.
        A closer look at Temlyakov's result shows that he constructs a \emph{linear} recovery operator.}
  \item{\label{dku}In \cite[Theorem~3]{DKU22} it is shown that
        \begin{equation*}
          \varrho_n(\CalF)_{L^2(\mu)}\leq\varrho_n^{\mathrm{lin}}(\CalF)_{L^2(\mu)}\leq C_p\lr{\frac{1}{n}\sum_{k\geq cn}d_n(\CalF)_{L^2(\mu)}^p}^{\frac{1}{p}}
        \end{equation*}
        for constants $C_p>0$ (depending only on $0<p<2$) and an absolute constant $c>0$.
        This result also gives a decay of $n^{-r}$ (up to logarithms) for $\CalF =\CalA^r_{\mathrm{mix}}(\T^d)$.
        Note that we need $r>\frac{1}{2}$ and $\frac{1}{r} < p < 2$ to apply this result.}
\end{enumerate}
\end{remark}

Note that the linear reconstruction methods mentioned in \Cref{rem:literature}\ref{temlyakov} and \ref{dku} also require the choice of certain parameters depending on the function class $\CalF$.
There the optimal subspace coming from the singular value decomposition (or the Kolmogorov widths) has to be known in advance to perform the weighted least squares method.
In our case we have to choose the noise parameter $\eta>0$ and properly determine the search space $J^\ast$ and its size $N$.

\subsubsection{Isotropic Wiener spaces on the torus}
Let us now come to the isotropic classes studied by DeVore and Temlyakov \cite[Section~6]{DeVoreTe1995}.
These are defined (up to minor typos in \cite{DeVoreTe1995}) as follows:
\begin{definition}
For $r>0$ and $0<p\leq 1$, we define
\begin{equation*}
\CalA^r_p(\T^d)
  \defeq \setcond{f\in L^1(\T^d)}{\sum_{k\in \Z^d}\lr{\abs{\widehat{f}(k)} (1+\mnorm{k}_{\ell^\infty(\Z^d)})^r}^p  < \infty}.
\end{equation*}
and
\begin{equation*}
\mnorm{f}_{\CalA^r_p(\T^d)} \defeq \lr{\sum_{k\in \Z^d}
 \lr{\abs{\widehat{f}(k)} (1+\mnorm{k}_{\ell^\infty(\Z^d)})^r}^p}^{\frac{1}{p}}.
\end{equation*}
\end{definition}
The quasi-Banach spaces $\CalA^r_p(\T^d)$ embeds compactly into $\bfun(\T^d)$ for all $r>0$ and $0<p\leq 1$.
If $p=1$, then $\CalA^r_p(\T^d)$ becomes a Banach space.
In \mbox{\cite[Theorem~6.1]{DeVoreTe1995}} it is proved
\begin{equation*}
  \sigma_n(\CalA^r_p(\T^d);\CalT^d)_{\bfun(\T^d)}
  \asymp n^{-\frac{r}{d}-\frac{1}{p}+\frac{1}{2}}.
\end{equation*}
Moreover, with a straightforward computation (similar as in \eqref{eq:SM}) we obtain
\begin{equation}\label{eq:E_Misotrop}
	E_{[M,M]^d \cap \Z^d}(\CalA^r_p(\T^d);\CalT^d)_{\bfun(\T^d)} \leq (1+M)^{-r}.
\end{equation}
We obtain the following corollary from \Cref{thm:improvedGeneralFourierBound}.
\begin{corollary}
  Let $r > 0$ and $0<p\leq 1$.
  Then there are universal constants $C,\widetilde{C} > 0$ and a constant $C_{r,p,d}>0$ such that 
  \begin{align*}
   & \varrho_{\ceil{ Cd\Log{d}n\Log{n}^2\log(M)}}(\CalA^r_p(\T^d))_{L^2(\T^d)}\\
    &\leq \widetilde{C}(\sigma_n(\CalA^r_p;\CalT^d)_{\bfun(\T^d)}+ (1+M)^{-r}).
  \end{align*}
  Choosing $M = \floor{n^{\frac{1}{d}+\frac{1}{pr}-\frac{1}{2r}}}$, we obtain
  \begin{equation*}
    \varrho_{\ceil{C_{r,p,d}n\Log{n}^3}}(\CalA^r_p(\T^d))_{L^2(\T^d)}
    \lesssim n^{-\frac{r}{d}-\frac{1}{p}+\frac{1}{2}}.
  \end{equation*}
  \end{corollary}
\subsubsection{Wiener spaces for systems of orthonormal polynomials}
We proceed with non-periodic Wiener classes $\CalA^r_{\alpha,p}(\pmone)$
taking the role of $\CalF$ in \Cref{thm:multivariate_algebraic}.
As in the periodic case, Wiener-type spaces are defined via a summability condition
involving the coefficients of expansions of elements with respect to the orthonormal system $\CalB$.
Here we focus on two important cases, namely Chebyshev polynomials of the first kind
$\CalB_{-\frac{1}{2}}$ and Legendre polynomials $\CalB_0$ as introduced in \Cref{chap:algebraic-polynomials}.
As before, we restrict ourselves to the case of univariate polynomials.

The following spaces have been considered in Rauhut and Ward \cite[Section~7]{RauhutWard2012}
for the special case $\alpha = 0$, i.e., the $L^2(\mu_0)$-normalized Legendre polynomials $(\gp^0_n)_{n \in\N_0}$.
Taking into account that $\mnorm{\gp^0_n}_{\bfun(\pmone)} \asymp \sqrt{n}$, a standard computation shows that the condition $r \geq \frac{1}{2}$ ensures 
\begin{equation*}
\CalA^r_{0,p}(\pmone) \hookrightarrow C(\pmone)\hookrightarrow \bfun(\pmone).
\end{equation*}
Note that in case $\alpha = -\frac{1}{2}$ the condition $r\geq 0$ is enough for such an embedding.
However, the embedding is compact if and only if $r>\frac{1}{2}+\alpha$. 

\begin{definition}\label{def:polynomial-wiener}
  Let $\alpha \in \setn{-\frac{1}{2},0}$, $r\geq 0$, and $0<p\leq 1$.
  We define the \emph{Wiener-type space}
\begin{equation*}
\CalA^r_{\alpha,p}(\pmone)
  \defeq \setcond{\sum_{n=0}^\infty\beta_n \gp^\alpha_n}{\sum_{n\in \N_0} \abs{\beta_n(1+n)^r}^p<\infty}.
\end{equation*}
\end{definition}
These spaces are Banach spaces if $p=1$ and quasi-Banach spaces in case $0<p<1$ with (quasi-)norm
\begin{equation*}
	\mnorm{f}_{\CalA^r_{\alpha,p}(\pmone)}\defeq\lr{\sum_{n\in \N_0} \abs{\skpr{f}{\gp^\alpha_n}_{L^2(\mu_\alpha)}(1+n)^r}^p}^{\frac{1}{p}}.
\end{equation*}
From \Cref{thm:multivariate_algebraic} we obtain an estimate for the sampling numbers of the Wiener classes for Chebyshev polynomials as part~\ref{chebyshev-wiener-sampling} of the following corollary.
The analogous part~\ref{legendre-wiener-sampling} for Legendre polynomials is postponed to \Cref{thm:legendre-wiener-sampling}.
\begin{corollary}\label{thm:polynomial-wiener}
\begin{enumerate}[label={(\roman*)},leftmargin=*,align=left,noitemsep]
\item{\label{legendre-wiener-sampling}Let $\alpha = 0$, $r > 0$, and $0 < p \leq 1$.
There are constants $C_{r,p}, \widetilde{C}_{r,p} > 0$ with 
        \begin{equation*}
          \varrho_{\ceil{C_{r,p}n\Log{n}^4}}(\CalA^r_{0,p}(\pmone))_{L^2(\mu_0)}
          \leq \widetilde{C}_{r,p}n^{-(r+\frac{1}{p}-1)}. 
        \end{equation*}
}
\item{\label{chebyshev-wiener-sampling}Let $\alpha = -\frac{1}{2}$, $r > 0$, and $0 < p \leq 1$.
There are constants $C_{r,p},\widetilde{C}_{r,p}>0$ such that 
        \begin{equation*}
          \varrho_{\ceil{C_{r,p}n\Log{n}^4}}(\CalA^r_{-\frac{1}{2},p}(\pmone))_{L^2(\mu_{-\frac{1}{2}})}
          \leq \widetilde{C}_{r,p}n^{-(r+\frac{1}{p}-\frac{1}{2})}.
        \end{equation*}
In particular, we obtain in the Banach space case $p = 1$ the rate 
        \begin{equation*}
          \varrho_{\ceil{C_{r,1}n\Log{n}^4}}(\CalA^r_{-\frac{1}{2},1}(\pmone))_{L^2(\mu_{-\frac{1}{2}})}
          \leq \widetilde{C}_{r}n^{-(r+\frac{1}{2})}.
        \end{equation*}
}        
\end{enumerate}
\end{corollary}

\begin{proof}
Here we only show part~\ref{chebyshev-wiener-sampling}.
For \ref{legendre-wiener-sampling}, see \Cref{thm:legendre-wiener-sampling}.
Via \Cref{thm:multivariate_algebraic}, it remains to estimate the quantities
$\sigma_n(\CalA^r_{-\frac{1}{2},p}(\pmone);\CalB_{-\frac{1}{2}})_{\bfun(\pmone)}$
and\\ $E_{\setn{0,\ldots,M}}(\CalA^r_{-\frac{1}{2},p}(\pmone);\CalB_{-\frac{1}{2}})_{\bfun(\pmone)}$.
From the definition of the Chebyshev polynomials $(\gp^{-\frac{1}{2}}_n)_{n\in\N_0}$ we have 
\begin{equation*}
\gp^{-\frac{1}{2}}_n(\cos(2\uppi\theta))=\sqrt{2}\cos(2\uppi n\theta)=\frac{\sqrt{2}}{2}(\ee^{2\uppi\ii n\theta}+\ee^{-2\uppi\ii n\theta}).
\end{equation*}
Thus, if $f\in \CalA^r_{-\frac{1}{2},p}(\pmone)$, we have $f \circ \cos(2\uppi\bullet) \in \CalA_p^r(\T)$ and 
\begin{equation*}
\mnorm{f\circ\cos(2\uppi\bullet)}_{\CalA^r_p(\T)}\leq \mnorm{f}_{\CalA^r_{-\frac{1}{2},p}(\pmone)}.
\end{equation*}
This gives
\begin{equation*}
	\sigma_n(\CalA^r_{-\frac{1}{2},p}(\pmone);\CalB_{-\frac{1}{2}})_{\bfun(\pmone)}
  \lesssim \sigma_n(\CalA^r_p(\T);\CalT)_{\bfun(\T)}
\end{equation*}
and
\begin{equation*}
	E_{\setn{0,\ldots,M}}(\CalA^r_{-\frac{1}{2},p}(\pmone);\CalB_{-\frac{1}{2}})_{\bfun(\pmone)}
  \lesssim E_{[-M,M]\cap \Z}(\CalA^r_p(\T);\CalT)_{\bfun(\T)}.
\end{equation*}
We know from \cite[Theorem~6.1]{DeVoreTe1995} that
\begin{equation*}
	\sigma_n(\CalA^r_p(\T);\CalT)_{\bfun(\T)} \lesssim n^{-(r+\frac{1}{p}-\frac{1}{2})}.
\end{equation*}
Together with \eqref{eq:E_Misotrop} and choosing $M$ appropriately gives the desired bound. 
\end{proof}
\begin{remark}
The bound in \Cref{thm:polynomial-wiener}\ref{legendre-wiener-sampling} is possibly not optimal.
We leave this as an open problem. 
 \end{remark}

For a comparison of the nonlinear sampling numbers and their linear counterparts in case of the Wiener classes for Chebyshev and Legendre polynomials, we compute the approximation numbers of these classes.
The result is based on a general insight about approximation  numbers of diagonal operators.

\begin{theorem}\label{thm:convex}
Let $(\gamma_n)_{n\in\N}$ be a non-increasing sequence of non-negative numbers, and
\begin{equation*}
D:\ell^1(\N)\to\ell^2(\N),\quad D((\beta_n)_{n\in\N})\defeq (\beta_n\gamma_n)_{n\in\N}.
\end{equation*}
Then we have for $0<p\leq 1$ and $n\in\N$ that 
\begin{equation*}
a_n(D^p)=a_n(D),
\end{equation*}
where $D^p:\ell^p(\N) \to \ell^2(\N)$ is the restriction of $D$ to $\ell^p(\N)$.
\end{theorem}
\begin{proof}
By the embedding
$\ell^p(\N)\hookrightarrow\ell^1(\N)$, it holds $a_n(D^p)\leq a_n(D)$.
For $T\in\N$, let
\begin{equation*}
\gamma_{n,T}=\begin{cases}\gamma_n&\text{if }n\leq T,\\0&\text{else}\end{cases}
\end{equation*}
and
\begin{equation*}
D_T:\ell^1(\N)\to\ell^2(\N),\quad D_T((\beta_n)_{n\in\N})\defeq (\beta_n\gamma_{n,T})_{n\in\N}.
\end{equation*}
Denote by $D_T^p$ the restriction of $D_T$ to $\ell^p(\N)$, and abbreviate the unit ball $B_{\ell^p(\N)}$ of $\ell^p(\N)$ by $B^p$.
\Cref{rem:widths} yields
\begin{equation*}
\lambda_n(D_T(B^p))_{\ell^2(\N)} = a_n(D_T^p) \leq a_n(D^p)
\end{equation*}
for all $n,T\in\N$.
Next, we observe that $D_T(B^1)$ is a non-empty compact convex subset of the finite-dimensional space $D_T(\ell^1(\N))$.
Its extreme points are finite in number and given by the non-zero points among $w_j\defeq(\gamma_j\delta_{j,n})_{n\in \N}$ and $-w_j$ for $j\in\FirstN{T}$, where $\delta_{j,n}$ is the Kronecker delta.
Thus $D_T(B^1)$ is the convex hull of the points $\pm w_1,\ldots,\pm w_T$, i.e.,
\begin{equation*}
D_T(B^1)=\co\setn{\pm w_1,\ldots,\pm w_T}
\end{equation*}
Since $\pm w_j\in D_T(B^p)\subset D_T(B^1)$, we know
\begin{align*}
D_T(B^1)=\co\setn{\pm w_1,\ldots,\pm w_T}\subset\co(D_T(B^p))\subset \co(D_T(B^1))=D_T(B^1).
\end{align*}
In particular, we have $\co(D_T(B^p)) = D_T(B^1)$.
Now \cite[Theorem~II.4.1]{Pinkus1985} and \Cref{rem:widths} yield for all $n\in \N$
\begin{align*}
a_n(D_T^p)&=\lambda_n(D_T(B^p))_{\ell^2(\N)} = \lambda_n(\co(D_T(B^p)))_{\ell^2(\N)}\\
&= \lambda_n(D_T(B^1))_{\ell^2(\N)} = a_n(D_T).
\end{align*}
Taking \cite[Theorem~11.11.7]{Pietsch1980} into account, we have
\begin{align*}
a_n(D^p)&\geq \sup_{T\in\N} a_n(D_T^p)\\
&=\sup_{T\in\N}a_n(D_T)\\
&=\sup_{T\in\N}\sup\setcond{\lr{\frac{h-n+1}{\sum_{j=1}^h\gamma_{j,T}^{-2}}}^{\frac{1}{2}}}{h\in\N,h\geq n}\\
&=\sup_{T\in\N}\sup\setcond{\lr{\frac{h-n+1}{\sum_{j=1}^h\gamma_j^{-2}}}^{\frac{1}{2}}}{h\in\N,T\geq h\geq n}\\
&=\sup\setcond{\lr{\frac{h-n+1}{\sum_{j=1}^h\gamma_j^{-2}}}^{\frac{1}{2}}}{h\in\N,h\geq n}\\
&= a_n(D).
\end{align*}
This concludes the proof.
\end{proof}
Now we are able to compute the approximation numbers of the embeddings of the Wiener-type spaces for Chebyshev and Legendre polynomials into $L^2(\mu_\alpha)$ by reducing it to approximation numbers of suitable diagonal operators.
\begin{corollary}\label{thm:approximation-chebyshev-wiener}
Let $\alpha \in \setn{-\frac{1}{2},0}$, $r>0$, and $0 < p \leq 1$.
Then
\begin{equation*}
a_n(\CalA^r_{\alpha,p}(\pmone))_{L^2(\mu_\alpha)}\geq Cn^{-r}
\end{equation*}
\end{corollary}
\begin{proof}
In the following commutative diagram
\begin{equation*}
\begin{tikzcd}
\ell^p(\N) \ar[swap]{d}{\cong} \ar{r}{D} & \ell^2(\N)\\
\CalA^r_{\alpha,p}(\pmone) \ar[hookrightarrow]{r}{\id}& L^2(\mu_\alpha) \ar[swap]{u}{\cong}
\end{tikzcd}
\end{equation*}
the two isometric isomorphisms indicated by $\cong$ are given by
\begin{align*}
(\beta_n)_{n\in\N}&\mapsto \sum_{n\in\N_0}\beta_{n+1}(1+n)^{-r}\gp^\alpha_n\\
\intertext{and}
\sum_{n\in\N_0}\beta_n \gp^\alpha_n&\mapsto (\beta_{n-1})_{n\in\N}.
\end{align*}
By \cite[Theorem~2.3.3]{Pietsch1987}, we find that 
\begin{equation*}
a_n(\CalA^r_{\alpha,p}(\pmone))_{L^2(\mu_\alpha)} = a_n(D),
\end{equation*}
where the diagonal operator $D:\ell^p \to \ell^2$ is given by $D((\beta_n)_{n\in\N})=(\beta_n n^{-r})_{n\in\N}$.
From \Cref{thm:convex} we obtain 
\begin{equation*}
a_n(D:\ell^p(\N)\to \ell^2(\N)) = a_n(D:\ell^1(\N)\to \ell^2(\N)).
\end{equation*}
Thanks to \cite[Theorem~11.11.7]{Pietsch1980}, we have a precise formula for these approximation numbers and obtain
\begin{align*}
&a_n(D:\ell^1(\N) \to  \ell^2(\N))\\
&= \sup\setcond{\lr{\frac{h-n+1}{\sum_{j=1}^h j^{2r}}}^{\frac{1}{2}}}{h\in\N, h\geq n}\\
&\gtrsim \sup\setcond{\lr{\frac{h-n+1}{h^{2r+1}}}^{\frac{1}{2}}}{h\in\N, h\geq n}\\
&\geq \lr{\frac{2n-n+1}{(2n)^{2r+1}}}^{\frac{1}{2}}\\
&\gtrsim n^{-r}.
\end{align*}
Here we used
\begin{equation*}
\sum_{j=1}^h j^{2r}\leq \int_1^{h+1}x^{2r}\dd x=\frac{(h+1)^{2r+1}-1}{2r+1}\lesssim h^{2r+1},
\end{equation*}
which in turn is a consequence of the monotonicity of $j\mapsto j^{2r}$.
\end{proof}
Comparing \Cref{thm:polynomial-wiener} and \Cref{thm:approximation-chebyshev-wiener}, the nonlinear sampling numbers show an acceleration of order $n^{-\frac{1}{p}+\frac{1}{2}}$ in the Chebyshev case and $n^{-\frac{1}{p}+1}$ in the Legendre case compared to the corresponding approximation numbers.
This shows that for quasi-Banach spaces, the acceleration is not limited to $n^{-\frac{1}{2}}$ as in the Banach space case, see \cite[Section~4.2]{NovakWo2008}.

\subsection{Sobolev spaces with mixed smoothness}
\label{chap:periodic-sobolev}

In this section, we substitute Sobolev spaces with bounded mixed derivatives
$H^r_{\mathrm{mix}}(\T^d)$ and $S^r_pW(\T^d)$ for $\CalF$ in \Cref{thm:improvedGeneralFourierBound}.

We begin with the corollary of \Cref{thm:improvedGeneralFourierBound} corresponding to the spaces
$H^r_{\mathrm{mix}}(\T^d)$ introduced in \Cref{def:wiener-sobolev}.
Note that we already know by \cite{DKU22} that 
\begin{equation*}
	d_n(H^r_{\mathrm{mix}}(\T^d))_{L^2(\T^d)}
  \asymp \varrho^{\mathrm{lin}}_n(H^r_{\mathrm{mix}}(\T^d))_{L^2(\T^d)}
  \asymp n^{-r}\Log{n}^{(d-1)r}.
\end{equation*}
Of course, we cannot expect to improve this bound with a nonlinear method,
see \cite[Theorem~4.8]{NovakWo2008}.
However, our approach gives a rather simple and semi-constructive method (we still have to use random information) for the following bound.
Note that \cite[Equation~(3.3)]{Temlyakov2021} gives a non-optimal and non-constructive upper bound on  $\varrho^{\mathrm{lin}}_n(H^r_{\mathrm{mix}}(\T^d))_{L^2(\T^d)}$ which is better than the one given by \Cref{thm:periodic-sobolev,lem:technical}.
\begin{corollary}\label{thm:periodic-sobolev}
Let $r > \frac{1}{2}$.
Then there exists a constant $C_{r,d} > 0$ such that
\begin{align*}
	\varrho_{\ceil{C_{r,d}n\Log{n}^3}}(H^r_{\mathrm{mix}}(\T^d))_{L^2(\T^d)}
  & \lesssim \sigma_n(H^r_{\mathrm{mix}}(\T^d);\CalT^d)_{\bfun(\T^d)}\\
	& \lesssim n^{-r}\Log{n}^{(d-1)r+\frac{1}{2}}. 	
\end{align*}
\end{corollary}

\begin{proof}
We simply plug in the bound \cite[Theorem~2]{Belinskii1989} for the quantity\linebreak
$\sigma_n(H^r_{\mathrm{mix}}(\T^d);\CalT^d)_{\bfun(\T^d)}$ appearing in the right-hand side
of \Cref{thm:improvedGeneralFourierBound}.
For the second quantity $E_{[-M,M]^d\cap\Z^d}(H^r_{\mathrm{mix}}(\T^d);\CalT^d)_{\bfun(\T^d)}$,
we use straightforward estimates, see \cite[Theorem~4.2.5]{DungTeUl2018}
and choose $M$ appropriately afterwards (see the proof of \Cref{cor:p} below). 
\end{proof}

Note that our method provides an alternative approach to \cite{KriegUl2021a} to disprove \cite[Conjecture~5.6.2]{DungTeUl2018}.
From the stated bound  in \Cref{thm:periodic-sobolev} it follows that 
compressed sensing methods yield better bounds than the ones coming from sparse grids (see \cite[Chapter~5]{DungTeUl2018}) when $d$ is large,
i.e., if $\frac{d-1}{2}>\frac{1}{2}+3r$.

Now we come to the $L^p$-counterparts $S^r_pW(\T^d)$ of the spaces $H^r_{\mathrm{mix}}(\T^d)$,
as defined in \Cref{chap:sobolev-appendix} for $r > 0$ and $1 < p < \infty$.
First of all, we may state similar results also in case of \emph{small smoothness},
i.e., for spaces $S^r_pW(\T^d)$ with $2<p<\infty$ and $\frac{1}{p}<r\leq \frac{1}{2}$,
as we will see below.
The square summability of the Kolmogorov numbers in $L^2(\T^d)$ is not required as in \cite{KriegUl2021b}.

In this particular range of parameters $\frac{1}{p} < r \leq \frac{1}{2}$ no constructive method is known so far as pointed out in the discussion at the end of \cite{TemlyakovUl2021, TemlyakovUl2022}.

The following general bound on the sampling numbers of $S^r_pW(\T^d)$ in $L^2(\T^d)$
is an extension of \Cref{thm:periodic-sobolev} which is the $p=2$ case.
\begin{corollary}\label{cor:p}
Let $1 < p < \infty$ and $\frac{1}{p}<r$.
Then there exists a constant $C_{r,p,d} > 0$ such that  
\begin{equation*}
	\begin{split}	
	\varrho_{\ceil{C_{r,p,d}n\Log{n}^3}}(S^r_pW(\T^d))_{L^2(\T^d)}
  & \lesssim \sigma_n(S^r_pW(\T^d);\CalT^d)_{\bfun(\T^d)}.
	\end{split}
\end{equation*}
\end{corollary}

\begin{proof}
We invoke \Cref{thm:improvedGeneralFourierBound} together with \cite[Theorem~4.2.5]{DungTeUl2018},
which is a straightforward consequence of the Littlewood--Paley characterization in the appendix,
see \Cref{FourierW}.
Note that the step hyperbolic cross $Q_m$ defined in \cite[Equation~(2.3.1)]{DungTeUl2018}
is contained in $[-2^m,2^m]^d \cap \Z^d$ and so
\begin{align*}
	E_{[-2^m,2^m]^d \cap \Z^d}(S^r_pW(\T^d);\CalT^d)_{\bfun(\T^d)}
  & \leq E_{Q_m}(S^r_pW(\T^d);\CalT^d)_{\bfun(\T^d)}\\
	& \lesssim 2^{-(r-\frac{1}{p})m}m^{(d-1)\kappa}.
\end{align*}
for some $\kappa > 0$.
Choosing $M \defeq 2^m$ with $n^{2r(r-\frac{1}{p})^{-1}}\leq M\leq 2n^{2r(r-\frac{1}{p})^{-1}}$ implies
\begin{equation*}
  E_{[-M,M]^d \cap \Z^d}(S^r_pW(\T^d))_{\bfun(\T^d)} \lesssim n^{-r}\lesssim \sigma_n(S^r_pW(\T^d))_{\bfun(\T^d)},
\end{equation*}
see \cite[Theorem~2]{Belinskii1989}.
\end{proof}

Upper bounds for the quantity $\sigma_n(S^r_pW(\T^d);\CalT^d)_{\bfun(\T^d)}$
appearing in \Cref{cor:p} are available under constraints on the parameters $p$ and $r$.
The case of small smoothness, i.e., $2<p<\infty$ and $\frac{1}{p}<r\leq \frac{1}{2}$
follows from \cite[Theorems~6.1,~6.2, and~6.3]{TemlyakovUl2022}.
Note that the bounds on $\varrho_n(S^r_pW(\T^d))_{L^2(\T^d)}$ given by \Cref{thm:p>2,lem:technical} are not optimal since \cite[Theorem~5.1]{TemlyakovUl2021} already gives better bounds on $\varrho^{\mathrm{lin}}_n(S^r_pW(\T^d))_{L^2(\T^d)}$.
\begin{corollary}\label{thm:p>2}
  Let $2 < p < \infty$.
  There exists a constant $C_{r,p,d} > 0$ such that the following holds.
  \begin{enumerate}[label={(\roman*)},leftmargin=*,align=left,noitemsep]
    \item{In case $\frac{1}{p}<r<\frac{1}{2}$, we have
          \begin{equation*}
            \varrho_{\ceil{C_{r,p,d} n\Log{n}^3}}(S^r_pW(\T^d))_{L^2(\T^d)}
            \lesssim n^{-r}\Log{n}^{(d-1)(1-r)+r}. 	
          \end{equation*}
}
  \item{In case $r=\frac{1}{2}$, we have
        \begin{align*}
&\varrho_{\ceil{C_{r,p,d}n\Log{n}^3}}(S^r_pW(\T^d))_{L^2(\T^d)}\\
&\lesssim n^{-r}\Log{n}^{(d-1)(1-r)+r}(\log\log(n+2))^{r+1}.   
        \end{align*}
}
  \end{enumerate}
\end{corollary}

From \Cref{cor:p} together with \cite[Theorem~2.9]{Temlyakov2015},
see also \cite[Theorem~7.5.2]{DungTeUl2018}, we obtain the following result
for the case $1 < p < 2$ and $\frac{1}{p} < r$.

\begin{corollary}\label{thm:p<2}
  Let $1 < p < 2$ and $\frac{1}{p} < r$.
  There exists a constant $C_{r,p,d} > 0$ such that
  \begin{align}
&\varrho_{\ceil{C_{r,p,d} n\Log{n}^3}}(S^r_pW(\T^d))_{L^2(\T^d)}\nonumber\\
&\lesssim \lr{\frac{\Log{n}^{d-1}}{n}}^{r-\frac{1}{p}+\frac{1}{2}}\Log{n}^{\frac{1}{2}-(d-1)(\frac{1}{p}-\frac{1}{2})}.\label{eq:p<2}
  \end{align}            
\end{corollary}

\begin{remark}
\Cref{thm:p<2} is remarkable since it shows a classical situation where the general sampling numbers decay faster than the approximation numbers (linear widths)
$a_n(S^r_pW(\T^d))_{L^2(\T^d)}$. 
Indeed, we have 
\begin{equation*}
	a_n(S^r_pW(\T^d))_{L^2(\T^d)}
  = d_n(S^r_pW(\T^d))_{L^2(\T^d)}
  \asymp \lr{\frac{\Log{n}^{d-1}}{n}}^{r-\frac{1}{p}+\frac{1}{2}},
\end{equation*}
see \cite[Theorem~4.5.1]{DungTeUl2018} and hence $\varrho_n = o(a_n)$ if $d$ is sufficiently large.
This follows from the bound in \eqref{eq:p<2} together with \Cref{lem:technical} in the appendix.   
In this situation the linear sampling numbers studied in \cite{ByrenheidUl2017,DKU22} have a slower decay (equal to the approximation numbers, see \cite[Equation~(1.8)]{ByrenheidUl2017}) than the nonlinear counterparts.
The main rate in \Cref{thm:p<2} is optimal and cannot be improved.
This follows from a (univariate) fooling function argument in \cite[Theorem~23]{NoTr06} and \cite[Theorem~1]{Dung2009}.
Let us finally emphasize that $\varrho_n = o(a_n)$ is a multivariate mixed smoothness effect which is not present in the isotropic situation as shown by Heinrich \cite{Heinrich2009}, see also \cite[Open Problem~18, page~123]{NovakWo2008}.  
\end{remark}

%% file: UnboundedONS.tex
\section{Beyond bounded orthonormal systems:\\Legendre polynomials}%
\label{sect:unbounded}

In this section, we show how our arguments can be used to derive bounds for the sampling numbers if the non-periodic Wiener classes $\CalA^r_{0,p}$ defined with respect to the orthonormalized Legendre polynomials $(L_n)_{n\in\N}$.
For these, \Cref{thm:MainAbstractResult} is not immediately applicable, since the Legendre polynomials are \emph{not} a bounded orthonormal system. 
Indeed, let the \emph{non-normalized} Legendre polynomials $(P_n)_{n\in\N_0}$ be defined as in \cite[(4.2.1)]{Lebedev1972} via Rodrigues's formula, i.e.,
\begin{equation*}
P_n(x)=\frac{1}{2^nn!}\frac{\dd^n}{\dd x^n}(x^2-1)^n,
\end{equation*}
from which it is easy to deduce that $P_n$ is a polynomial of degree $n$.
Then \cite[(4.5.1) and (4.5.2)]{Lebedev1972} show that
\begin{equation*}
\int_{-1}^1 P_n(x)P_m(x)\dd x=\frac{2}{2n+1}\delta_{n,m}
\end{equation*}
so that $\CalB_0=(L_n)_{n\in\N_0}$ with
\begin{equation*}
L_n=\sqrt{n+\frac{1}{2}}P_n
\end{equation*}
is a family of polynomials with $\deg(L_n)=n$ and
\begin{equation*}
\int_{-1}^1 L_n(x)L_m(x)\dd x=\delta_{n,m},
\end{equation*}
called the \emph{orthonormalized Legendre polynomials}.\footnote{With the notation from \Cref{chap:algebraic-polynomials}, we have $p^0_n=\sqrt{2}L_n$.}
As shown in \cite[(4.4.2) and (4.2.7)]{Lebedev1972} we have 
\begin{equation}
P_n(1)=1\quad\text{and}\quad\abs{P_n(x)}\leq 1\text{ for }x\in \pmone
\end{equation}
which easily shows for $\Omega=(-1,1)$ that
\begin{equation}
\mnorm{L_n}_{\bfun(\pmone)}=\sqrt{n+\frac{1}{2}}\stackrel{n\to\infty}{\longrightarrow}\infty.
\end{equation}
We will circumvent this issue by a \enquote{preconditioning step} as in \cite{RauhutWard2012}.
Precisely, we introduce the weight function
\begin{equation}\label{eq:legendre-weight}
w:\pmone\to (0,\infty),\quad w(x)=\sqrt{\uppi}(1-x^2)^{\frac{1}{4}}
\end{equation}
and the Borel measure $\mu$ on $\pmone$ given by its density as
\begin{equation*}
\dd\mu(x)=\frac{1}{\uppi}\frac{1}{\sqrt{1-x^2}}\dd x.
\end{equation*}
We then have
\begin{equation}
\mu(\pmone)=\frac{1}{\uppi}\int_{-1}^1\frac{1}{\sqrt{1-x^2}}\dd x=1
\end{equation}
so that $\mu$ is a probability measure.
Next, we note (with the Lebesgue measure $\lambda$ on $\pmone$) that\footnote{With the notation from \Cref{chap:algebraic-polynomials}, we have $\lambda=2\mu_0$.}
\begin{equation*}
\Phi: L^2(\lambda)\to L^2(\mu),\quad f\mapsto w\cdot f
\end{equation*}
is an isometric isomorphism which follows from the computation
\begin{align}
\mnorm{w\cdot f}_{L^2(\mu)}^2&=\int_{-1}^1\abs{w(x)f(x)}^2\frac{1}{\uppi}\frac{1}{\sqrt{1-x^2}}\dd x\nonumber\\
&=\int_{-1}^1\abs{f(x)}^2\dd x=\mnorm{f}_{L^2(\lambda)}^2\label{eq:isometry}
\end{align}
Thus, if we define $(b_n)_{n\in\N_0}$ by
\begin{equation*}\label{eq:preconditioned-legendre}
b_n=w\cdot L_n,\quad\text{i.e.,}\quad b_n(x)=\sqrt{\uppi}(1-x^2)^{\frac{1}{4}}L_n(x)
\end{equation*}
then we have on the one hand that 
\begin{equation*}
\int_{-1}^1 b_n(x)b_m(x)\dd \mu(x)=\int_{-1}^1 L_n(x)L_m(x)\dd x=\delta_{n,m}
\end{equation*}
so that $(b_n)_{n\in\N_0}$ is an orthonormal system in $L^2(\mu)$.
On the other hand, \cite[\textsection I.7 (6.11)]{Freud1969} implies that
\begin{equation*}
\abs{b_n(x)}\leq 4\sqrt{\uppi}\qquad\text{for all }n\in\N_0\text{ and }x\in \pmone 
\end{equation*}
so that $\CalB_0^\natural=(b_n)_{n\in\N_0}$ is a bounded orthonormal system with $K(\CalB_0^\natural)=4\sqrt{\uppi}$.

The remaining issue is that the projection operator associated to the Legendre polynomials is not well compatible with the \enquote{preconditioning step} of multiplying with the weight $w$.
We circumvent this second issue by invoking \Cref{thm:restrictedbestnterm}, which does not rely on projection operators, and obtain the following general result for sampling numbers in terms of restricted best $n$-term approximation errors using Legendre polynomials.
\begin{theorem}\label{thm:weighted-legendre}
With $w$ as in \Cref{eq:legendre-weight}, consider the space
\begin{equation*}
\bfun_w(\pmone)\defeq \setcond{f:\pmone\to\CC}{f\text{ measurable and }\mnorm{w\cdot f}_{\bfun(\pmone)}<\infty}
\end{equation*}
with norm $\mnorm{f}_{\bfun_w(\pmone)}\defeq\mnorm{w\cdot f}_{\bfun(\pmone)}$.
Let $\CalF\hookrightarrow\bfun_w(\pmone)$ be a quasi-normed space, and let $n\in\N$ and $J^\ast\subset \N_0$ and $N\defeq\card{J^\ast}$.
Then, with the Lebesgue measure $\lambda$ on $\pmone$, we have
\begin{equation*}
\varrho_{\ceil{Cn\Log{n}^3\Log{N}}}(\CalF)_{L^2(\lambda)}\leq \widetilde{C}\sigma_{n,J^\ast}(\CalF;\CalB_0)_{\bfun_w(\pmone)}
\end{equation*}
for all $n\in\N$, where $\CalB_0=(L_n)_{n\in\N}$ is the family of orthonormalized Legendre polynomials and $C$, $\widetilde{C}$ are universal constants.
\end{theorem}
\begin{proof}
Let $\mu$ and $\CalB_0^\natural=(b_n)_{n\in\N_0}$ as discussed as the beginning of \Cref{sect:unbounded}.
Then $\CalB_0^\natural\subset L^2(\mu)$ is a bounded orthonormal system with $K\defeq K(\CalB_0^\natural)=4\sqrt{\uppi}$.
We also note that $\mnorm{L_j}_{\bfun_w(\pmone)}=\mnorm{w\cdot L_j}_{\bfun(\pmone)}=\mnorm{b_j}_{\bfun(\pmone)}\leq 4\sqrt{\uppi}$ and thus $L_j\in \bfun_w(\pmone)$ for all $j\in \N_0$.
Let $\CalF^\natural \defeq\setcond{w\cdot f}{f\in \CalF}$, equipped with the quasi-norm $\mnorm{g}_{\CalF^\natural}\defeq\mnorm{\frac{1}{w}\cdot g}_\CalF$.
We then have 
\begin{equation*}
\mnorm{g}_{\bfun(\pmone)}=\mnorm{\frac{1}{w}\cdot g}_{\bfun_w(\pmone)}\leq\mnorm{\frac{1}{w}\cdot g}_\CalF=\mnorm{g}_{\CalF^\natural}
\end{equation*}
for $g\in\CalF^\natural$ and hence $\CalF^\natural\hookrightarrow \bfun(\pmone)$.
Thus by applying \Cref{thm:restrictedbestnterm} to $\CalB_0^\natural$ and $\CalF^\natural$ in place of $\CalB$ and $\CalF$, we obtain for
\begin{equation*}
m=\ceil{C\cdot 16\uppi\cdot n\cdot\Log{n}^3\cdot \Log{N}}
\end{equation*}
certain sampling points $t_1,\ldots,t_m\in \pmone$ and a reconstruction operator 
\begin{equation*}
R^\natural:\CC^m\to\linspan\setcond{b_j}{j\in J^\ast}\subset \bfun(\pmone)
\end{equation*}
satisfying
\begin{equation}\label{eq:natural}
\sup_{g\in B_{\CalF^\natural}}\mnorm{g-R^\natural(g(t_1),\ldots,g(t_m))}_{L^2(\mu)}\leq\widetilde{C}\sigma_{n,J^\ast}(\CalF^\natural;\CalB_0^\natural)_{\bfun(\pmone)}.
\end{equation}
Now define
\begin{equation*}
R:\CC^m\to\linspan\setcond{L_j}{j\in J^\ast},\quad (y_1,\ldots,y_m)\mapsto\frac{1}{w}\cdot R^\natural(w(t_1)y_1,\ldots,w(t_m)y_m).
\end{equation*}
For $f\in B_\CalF$ we then have $f^\natural\defeq w\cdot f\in\CalF^\natural$ with $f^\natural\in B_{\CalF^\natural}$.
In view of \eqref{eq:natural} and \eqref{eq:isometry}, we then get
\begin{align*}
\mnorm{f-R(f(t_1),\ldots,f(t_m))}_{L^2(\lambda)}&=\mnorm{w\cdot f-w\cdot R(f(t_1),\ldots,f(t_m))}_{L^2(\mu)}\\
&=\mnorm{f^\natural- R^\natural(f^\natural(t_1),\ldots,f^\natural(t_m))}_{L^2(\mu)}\\
\leq\widetilde{C} \sigma_{n,J^\ast}(\CalF^\natural;\CalB_0^\natural)_{\bfun(\pmone)}.
\end{align*}
Therefore,
\begin{equation*}
\varrho_{\ceil{16C\uppi n\Log{n}^3\Log{N}}}(\CalF)_{L^2(\lambda)}\leq\widetilde{C}\sigma_{n,J^\ast}(\CalF^\natural;\CalB_0^\natural)_{\bfun(\pmone)}.
\end{equation*}
Finally, let $g\in B_{\CalF^\natural}$.
By definition, this means $f\defeq \frac{1}{w}\cdot g \in B_{\CalF}$.
Since $\Sigma_n\cap V_{J^\ast}$ (defined with respect to $\CalB_0=(L_j)_{j\in\N_0}$) is a closed subset of the finite-dimensional vector space $V_{J^\ast}$ (as a finite union of subspaces), there exists $h\in \Sigma_n\cap V_{J^\ast}$ with $\mnorm{f-h}_{\bfun_w(\pmone)}\leq\sigma_{n,J^\ast}(\CalF;\CalB_0)_{\bfun_w(\pmone)}$.
We have $h=\sum_{j\in J^\ast}x_jL_j$ with $\mnorm{x}_{\ell^0(N)}\leq n$.
Therefore
\begin{equation*}
h^\natural\defeq w\cdot h=\sum_{j\in J^\ast}x_jw\cdot L_j=\sum_{j\in J^\ast}x_jb_j
\end{equation*}
and thus
\begin{align*}
\sigma_{n,J^\ast}(g;\CalB_0^\natural)_{\bfun(\pmone)}&\leq \mnorm{g-h^\natural}_{\bfun(\pmone)}=\mnorm{\frac{1}{w}\cdot g-\frac{1}{w}\cdot h^\natural}_{\bfun_w(\pmone)}\\
&=\mnorm{f-h}_{\bfun(\pmone)}\leq\sigma_{n,J^\ast}(\CalF;\CalB_0)_{\bfun_w(\pmone)}.
\end{align*}
Since this holds for all $g\in B_{\CalF^\natural}$, we see
\begin{equation*}
\sigma_{n,J^\ast}(\CalF^\natural;\CalB_0^\natural)_{\bfun((-1,1))}\leq \sigma_{n,J^\ast}(\CalF;\CalB_0)_{\bfun_w(\pmone)}
\end{equation*}
and this completes the proof.
\end{proof}

Finally, we derive bounds for the sampling numbers of the Wiener-type spaces $\CalA^r_{0,p}(\pmone)$ introduced in \Cref{def:polynomial-wiener}.
Note that for every $f \in \CalA^r_{0,p}$ with $p\leq 1$ and $r> 0$ we have that $x\mapsto (1-x^2)^{1/4}f(x)$ is a continuous and bounded function on $[-1,1]$ and thus admits evaluation.
Unlike as in \Cref{thm:multivariate_algebraic} the following result only needs the restriction $r>0$ (see small paragraph at the end of the present section). 
 
\begin{corollary}\label{thm:legendre-wiener-sampling}
Let $r>0$ and $0<p\leq 1$.
There are constants $C_{r,p}, \widetilde{C}_{r,p}>0$ such that
\begin{equation*}
\varrho_{\ceil{C_{r,p}n\Log{n}^4}}(\CalA^r_{0,p}(\pmone))_{L^2(\lambda)}\leq \widetilde{C}_{r,p} n^{-(r+\frac{1}{p}-1)}
\end{equation*}
for all $n\in\N$.
\end{corollary}
\begin{proof}
Given $n\in\N$, we choose $J^\ast\defeq \setn{0,\ldots,M}$, where $M$ will be determined below.
Denote by $T_M$ the partial sum operator for the expansion $f = \sum_{j=0}^{\infty} \beta_j L_j$ defined as
\begin{equation*}
	T_M f \defeq \sum\limits_{j\in J^\ast} \beta_j L_j.
\end{equation*}
It clearly holds
\begin{align}
\mnorm{f-T_Mf}_{\bfun_w(\pmone)}&\leq \mnorm{\sum_{j\notin J^\ast}\beta_jL_j}_{\bfun_w(\pmone)} \lesssim \sum_{j\notin J^\ast}\abs{\beta_j}\label{eq:A_1}\\
&\lesssim (1+M)^{-r}\sum_{j\notin J^\ast}(1+j)^r\abs{\beta_j}.\nonumber
\end{align}
Choosing now $M\defeq\floor{n^{1+\frac{1}{pr}-\frac{1}{r}}}$ we obtain
\begin{equation*}
	\mnorm{f-T_Mf}_{\bfun_w(\pmone)}\lesssim n^{-(r+\frac{1}{p}-1)}\|f\|_{\CalA^r_{0,1}(\pmone)}.
\end{equation*}
Given $f\in B_{\CalA^r_{0,p}(\pmone)}$ we may estimate for any $n$-term sum $g$ of Legendre polynomials  
\begin{align*}
	\mnorm{f-T_Mg}_{\bfun_w(\pmone)} &\leq \mnorm{f-T_Mf}_{\bfun_w(\pmone)} + \mnorm{T_M(f-g)}_{\bfun_w(\pmone)}\\
	&\lesssim n^{-(r+\frac{1}{p}-1)} + \mnorm{f-g}_{\CalA^0_{0,1}(\pmone)}.
\end{align*}
This implies 
\begin{equation}\label{eq:J*}
	\sigma_{n,J^\ast}(\CalA^r_{0,p}(\pmone);\CalB_0)_{\bfun_w(\pmone)} \lesssim n^{-(r+\frac{1}{p}-1)}+\sigma_{n}(\CalA^r_{0,p}(\pmone);\CalB_0)_{\CalA^0_{0,1}(\pmone)}.
\end{equation}
Applying Stechkin's inequality \cite[Lemma~7.4.1]{DungTeUl2018} gives
\begin{equation*}
	\sigma_n(\CalA^r_{0,p}(\pmone);\CalB_0)_{\CalA^r_{0,1}(\pmone)}
  \lesssim n^{-(\frac{1}{p}-1)}.
\end{equation*}
In addition, we get from \eqref{eq:A_1} 
\begin{equation*}
	\sigma_n(\CalA^r_{0,1}(\pmone);\CalB_0)_{\CalA^0_{0,1}} \lesssim n^{-r}.
\end{equation*}
Combining both estimates using \Cref{lem:factor} with $n_1=n_2 = n$ together with the commutative diagram  
\begin{equation*}
\begin{tikzcd}
\CalA^r_{0,p}(\pmone) \ar[hookrightarrow]{rr}{} \ar[hookrightarrow,swap]{rd}{} && A^0_{0,1}(\pmone) \\
& \CalA^r_{0,1}(\pmone) \ar[hookrightarrow,swap]{ru}{} &              
\end{tikzcd}
\end{equation*}
yields 
\begin{equation}\label{eq:sigma}
	\sigma_{2n}(\CalA^r_{0,p}(\pmone);\CalB_0)_{\CalA^0_{0,1}(\pmone)}
  \lesssim n^{-(r +\frac{1}{p}-1)}.
\end{equation}
Because of
\begin{equation*}
\log(\card{J^\ast})=\log\lr{1+\floor{n^{1+\frac{1}{pr}-\frac{1}{r}}}}\lesssim \Log{n}
\end{equation*}
the claim now follows from \Cref{thm:weighted-legendre}, \eqref{eq:J*}, and \eqref{eq:sigma}.
\end{proof}
Note that for \Cref{thm:legendre-wiener-sampling} we only need $r>0$ whereas for \Cref{thm:multivariate_algebraic} the stronger constraint $r>\frac{1}{2}$ is required.

\begin{proof}[Proof of \Cref{thm:multivariate_algebraic}]
It remains to show the assertion in the Legendre case $\alpha=0$.
As $w$ is bounded, we have $\CalF\hookrightarrow\bfun(\pmone)\hookrightarrow \bfun_w(\pmone)$.
We use \Cref{thm:weighted-legendre} in conjuction with \Cref{lem:conn} with $J=\setn{0,\ldots,M}$ and $J^\ast=\setn{0,\ldots,3M}$ for the quasi-projections described in \Cref{chap:algebraic-polynomials}.
We obtain $\kappa=1$ and $\tau\leq 4$, whence
\begin{align*}
&\varrho_{\ceil{Cn\Log{n}^3\log(4M)}}(\CalF)_{L^2(\lambda)}\\
&\leq \varrho_{\ceil{Cn\Log{n}^3\Log{3M}}}(\CalF)_{L^2(\lambda)}\\
&\leq \widetilde{C}\sigma_{n,J^\ast}(\CalF;\CalB_0)_{\bfun_w(\pmone)}\\
&=\widetilde{C}\lr{(1+\tau)E_{\setn{0,\ldots,M}}(f;\CalB_0)_{\bfun_w(\pmone)}+\tau\sigma_n(f;\CalB_0)_{\bfun_w(\pmone)}}\\
&\leq 5\widetilde{C}\lr{E_{\setn{0,\ldots,M}}(f;\CalB_0)_{\bfun_w(\pmone)}+\sigma_n(f;\CalB_0)_{\bfun_w(\pmone)}}\\
&\leq 5\widetilde{C}c\lr{E_{\setn{0,\ldots,M}}(f;\CalB_0)_{\bfun(\pmone)}+\sigma_n(f;\CalB_0)_{\bfun(\pmone)}}
\end{align*}
with the constants $C,\widetilde{C}>0$ from \Cref{thm:weighted-legendre}.
Observing $\log(4M)\lesssim\log(M)$ takes account for the index on the left-hand side.
In the last step, we utilized that the embedding $\bfun(\pmone)\hookrightarrow \bfun_w(\pmone)$ implies the existence of a constant $c>0$ such that $\mnorm{f}_{\bfun_w(\pmone)}\leq c\mnorm{f}_{\bfun(\pmone)}$ for all $f\in \bfun(\pmone)$.
\end{proof}

%% file: MixedSobolev.tex
\section{Sobolev spaces with mixed smoothness}%
\label{chap:sobolev-appendix}

We introduce Sobolev spaces with mixed smoothness.
These spaces have a relevant history in the former Soviet Union, see for instance \cite[Chapter~3]{DungTeUl2018} and the references given there.
They also play a significant role in discrepancy theory and numerical integration, see \cite[Chapter~5]{DungTeUl2018}.
Finally, these spaces turn out to be useful for the analysis of eigenfunctions of certain quantum mechanical operators, see \cite{Ys10}.
Define for $x \in \T$ and $r > 0$ the univariate Bernoulli kernel
\begin{equation*}
  F_r(x)
  \defeq  1 + 2\sum_{k=1}^\infty k^{-r}\cos (2\uppi kx) =
  \sum_{k\in\Z}\max\setn{1,\abs{k}}^{-r}\ee^{-2\uppi\ii kx}
\end{equation*}
and define the multivariate Bernoulli kernels as the corresponding tensor products
\begin{equation*}\label{Bernoulli}
  F_r(x)
  \defeq \prod_{j=1}^dF_r(x_j),
  \quad x=(x_1,\ldots,x_d)\in \T^d.
\end{equation*}

\begin{definition} \label{def2Sob}
Let $r > 0$ and $1 < p < \infty$.
Then $S^r_pW(\T^d)$ is defined as the normed space of all elements $f\in L^p(\T^d)$ which can be written as
\begin{equation*}
  f
  = F_r\ast \varphi
  \defeq \int_{\T^d}F_r(\bullet-y)\varphi(y)\dd y
\end{equation*}
for some $\varphi \in L^p(\T^d)$, equipped with the norm
$\mnorm{f}_{S^r_pW(\T^d)}\defeq \mnorm{\varphi}_{L^p(\T^d)}$.
\end{definition}
In other words, a function $f\in L^p(\T^d)$ is an element of $S^r_pW(\T^d)$ if and only if
\begin{equation*}
g_f\defeq \sum_{k\in\Z^d}\prod_{j=1}^d\max\setn{1,\abs{k_j}}\widehat{f}(k)\ee^{2\uppi\ii \skpr{k}{\bullet}}\in L^p(\T^d),
\end{equation*}
and then $\mnorm{f}_{S^r_pW(\T^d)}\defeq \mnorm{g_f}_{L^p(\T^d)}$.

The spaces $S^r_pW(\T^d)$ are Banach spaces for all $r > 0$ and $1 < p < \infty$.
They are Hilbert spaces exactly for $p=2$.
In that case, it is well-known that $S^r_2W(\T^d) = H^r_{\mathrm{mix}}(\T^d)$ with equivalent norms.
This fact is easily implied by the Littlewood--Paley characterization of the spaces $S^r_pW(\T^d)$.
Here we use building blocks with frequency support in dyadic rectangles.
For $s\in \N_0^d$, we set 
\begin{equation*}\label{defdelta}
  \delta_{s}(f,x)
  \defeq \sum_{k \in \varrho(s)} \widehat{f}(k)\ee^{2\uppi \ii \skpr{k}{x}}
\end{equation*}
using 
\begin{equation*}
  \varrho(s)
  \defeq  \setcond{ k\in \Z^d}
                  {\floor{2^{s_j-1}} \leq \abs{k_j} < 2^{s_j} \fall j=1,\ldots,d }.
\end{equation*}
\begin{lemma}\label{FourierW}
  If $1 < p < \infty$ and $r > 0$ then the norm $\mnorm{f}_{S^r_pW(\T^d)}$
  is equivalent to the Littlewood--Paley type norm
  \begin{equation*}
    \mnorm{f}_{S^r_pW(\T^d)}
    \asymp \mnorm{
             \lr{\sum_{s \in \N_0^d}
             2^{2r \mnorm{s}_{\ell^1(\N_0^d)}}
             \abs{\delta_s(f,\bullet)}^2}^{\frac{1}{2}}
           }_{L^p(\T^d)}.
  \end{equation*}
\end{lemma}

In order to have access to function values we use the restriction $r > \frac{1}{p}$ which implies that every equivalence class $f \in S^r_pW(\T^d)$ contains a continuous periodic function, see \cite[Lemma~3.4.1(iii) and~3.4.3]{DungTeUl2018}.
Moreover, the embedding 
\begin{equation*}
	S^r_pW(\T^d) \hookrightarrow \bfun(\T^d)
\end{equation*}
is then compact.

\section{Some technical lemmas}
\label{sect:lech_lemma}

\begin{lemma}\label{lem:gelfand-sampling}
Let $\CalF$ be Banach space of functions $\Omega \to\CC$ which is continuously embedded into $\bfun(\Omega)\hookrightarrow X$ where $X$ is another Banach space.
Then $c_n(\CalF)_X\leq\varrho_n(\CalF)_X$ for all $n\in\N$.
\end{lemma}

\begin{proof}
Let $\eps>0$.
There exist $t_1,\ldots,t_n\in\Omega$, and $R:\CC^n\to X$ such that
\begin{equation*}
\sup_{f\in B_\CalF}\mnorm{f-R(f(t_1),\ldots,f(t_n))}\leq \varrho_n(\CalF)_X+\eps
\end{equation*}
Set $L\defeq\setcond{f\in\CalF}{f(t_1)=\ldots =f(t_n)=0}$.
Then $L$ is a closed subspace of $\CalF$ (this uses that $\CalF\hookrightarrow \bfun(\Omega)$) and $\codim(L)\leq n$.
Now let $f\in L$ with $\mnorm{f}_\CalF\leq 1$.
Then
\begin{align*}
2\mnorm{f}_X&=\mnorm{f-(-f)}_X\\
&=\mnorm{f-R(f(t_1),\ldots,f(t_n))-((-f)-R(-f(t_1),\ldots,-f(t_n)))}_X\\
&\leq\mnorm{f-R(f(t_1),\ldots,f(t_n))}_X+\mnorm{(-f)-R(-f(t_1),\ldots,-f(t_n))}_X\\
&\leq 2(\varrho_n(\CalF)_X+\eps).
\end{align*}
Passing $\eps\downarrow 0$ and taking the supremum over $f\in L\cap B_\CalF$ yields the assertion.
\end{proof}

\begin{lemma}\label{thm:approximation-numbers-linear-widths}
Let $\CalF$ be a quasi-normed space, let $H$ a Hilbert space, and let $T:\CalF\to H$ be bounded and linear.
Then
\begin{equation*}
	a_n(T:\CalF \to H) = \lambda_n(T(B_\CalF))_H
\end{equation*}
for all $n\in\N$.
\end{lemma}
\begin{proof}
First, we show \enquote{$\geq$}. Let $\eps >0$.
Choose a bounded linear operator $A:\CalF \to H$ with $\rank A \leq n$ such that 
\begin{equation*}
	\sup_{f\in B_\CalF}
               \mnorm{Tf-Af}_H \leq a_n(T:\CalF \to H)+\eps.
\end{equation*}
Consider the orthogonal projection $P_L:H\to H$ onto $L = R(A)$, i.e., the range of $A$.
Then it holds for any $f \in B_\CalF$
\begin{equation*}
	\mnorm{Tf-P_L(Tf)}_H \leq \mnorm{Tf-Af}_H\leq a_n(T:\CalF \to H)+\eps
\end{equation*}
and hence $\lambda_n(T(B_\CalF))_H \leq a_n(T:\CalF \to H)+\eps$. 

Second, we show \enquote{$\leq$}.
Let $\eps>0$.
Choose a bounded linear operator $A:H\to H$ with $\rank A\leq n$ such that 
\begin{equation*}
	\sup_{f\in B_\CalF}\mnorm{Tf-A(Tf)}_H \leq \lambda_n(T(B_\CalF))_H+\eps.
\end{equation*}
Then we again put $L = R(A)$ and consider the orthogonal projection $P_L$ onto this subspace.
Then clearly 
\begin{equation*}
	\mnorm{Tf-P_L(Tf)}_H \leq \mnorm{Tf-A(Tf)}_H
\end{equation*}
for all $f\in\CalF$ and hence 
\begin{equation*}
	\sup_{f\in B_\CalF}\mnorm{Tf-(P_L \circ T)(f)}_H \leq \lambda_n(T(B_\CalF))_H+\eps,
\end{equation*}
which implies $a_n(T:\CalF \to H) \leq \lambda_n(T(B_\CalF))_H+\eps$.
\end{proof}

\begin{lemma}\label{lem:technical}
Let $(a_n)_{n\in\N}$ denote a non-increasing sequence of non-negative real numbers satisfiying
\begin{equation*}
	a_{\ceil{c_1 n\log(n)^{\alpha}}} \leq c_2n^{-r}\log(n)^{\beta},\quad n\in \N,\, n\geq 3,
\end{equation*}
for some $c_2,r>0$, $\alpha,\beta \geq 0$ and $c_1\geq 1$.
Then it holds for $m \geq \ceil{ 3c_1\log(3)^{\alpha}}$ that 
\begin{equation*}
	a_m \leq c_2(4c_12^{\alpha})^rm^{-r}\log(m)^{\beta+\alpha\cdot r}.
\end{equation*}
\end{lemma}
 \begin{proof}
 For $n\in \N$, $n\geq 3$ we define 
\begin{equation*}
	m_n\defeq \ceil{ c_1n \log(n)^\alpha}. 
\end{equation*}
 Then we have $m_n \geq \ceil{ 3c_1\log(3)^{\alpha}}, \log(m_n)  \geq \log(n)$ and $m_n \leq 2c_1n \log(n)^\alpha$.
 This implies
 \begin{equation*}
   \begin{split}
	a_{m_n} &\leq c_22^rc_1^rm_n^{-r}\log(n)^{\alpha\cdot r+\beta}\\
	&\leq c_22^rc_1^rm_n^{-r}\log(m_n)^{\alpha\cdot r+\beta}.
   \end{split}	 
 \end{equation*}
 Now, for $m\geq \ceil{ 3c_1\log(3)^{\alpha}}$ there exists a number $n\in \N$, $n\geq 3$ such that $m_n\leq m < m_{n+1}$.
 From $m_{n+1}= \ceil{ c_1(n+1)\log(n+1)^{\alpha}}$, we obtain 
\begin{equation*}
m \leq 2c_12^{\alpha}n\log(n)^{\alpha}\leq 2^{\alpha+1}m_n
\end{equation*}
 This finally yields  
\begin{equation*}
	a_m \leq a_{m_n}\leq c_2(2c_12^{\alpha+1})^rm^{-r}\log(m)^{\alpha\cdot r + \beta}
\end{equation*}
 for $m$ in the given range. 
\end{proof}

\begin{remark}
The argument in the previous lemma produces additional constants $m\geq \ceil{3c_1\log(3)^{\alpha}}$ and 
\begin{equation*}
	a_m \leq C_{\alpha,r,c_1}c_2m^{-r}\log(m)^{\alpha\cdot r+\beta},
\end{equation*}
which only depend on $c_1,\alpha$ and $r$, but not on $\beta$.
\end{remark}